\documentclass[a4paper,12pt,twoside,leqno,tbtags]{article}
\usepackage{amsmath}
\usepackage{amsfonts}
\usepackage{amssymb}
\usepackage{latexsym}
\usepackage{epsfig}
\usepackage{graphicx}
\usepackage{oldgerm}
\usepackage{amsthm}
\usepackage{color}

\theoremstyle{definition}
\newtheorem{dfn}{Definition}[section]

\newtheorem{ex}{Example}[section]

\newtheorem{thm}{Theorem}[section]
\newtheorem{lem}[thm]{Lemma}
\newtheorem{cor}[thm]{Corollary}
\newtheorem{prop}[thm]{Proposition}
\newtheorem{rem}{Remark}[section]

\newcommand{\G}{{\sf G}}
\newcommand{\x}{\mathbf{x}}
\newcommand{\y}{\mathbf{y}}
\newcommand{\mm}{\ell}
\newcommand{\nnn}{n}
\newcommand{\nn}{{\sf n}}
\newcommand{\qq}{q}
\newcommand{\C}{\mathbb{C}}
\newcommand{\K}{{\sf K}}
\newcommand{\Lsf}{{\sf L}}
\newcommand{\wbar}{\overline{w}}
\newcommand{\bra}{\langle}
\newcommand{\ket}{\rangle}

\newcommand{\var}{\mathrm{Var}}

\newcommand{\la}{\lambda}
\newcommand{\Prob}{{\mathbb P}}
\newcommand{\z}{\mathbb{Z}}
\newcommand{\n}{\mathbb{N}}
\newcommand{\E}{{\mathbb E}}
\newcommand{\bcal}{{\mathcal B}}
\newcommand{\Hcal}{{\mathcal H}}
\newcommand{\p}{{\mathcal P}}
\newcommand{\bb}{\mathbf{b}}
\newcommand{\s}{\mathbf{s}}
\newcommand{\oo}{\mathbf{o}}

\newcommand{\gsf}{{\sf g}}
\newcommand{\xis}{{\sf s}}
\newcommand{\xit}{{\sf t}}
\newcommand{\xiu}{{\sf u}}
\newcommand{\gggn}{\gsf ^{\nn }} 
\newcommand{\gggm}{\gsf ^{\mm }} 
\newcommand{\ZZ}{\mathcal{Z}}

\newcommand{\ind}{\mathbf{1}}
\newcommand{\eps}{\epsilon}
\newcommand{\nutilde}{\tilde{\nu}}
\newcommand{\radon}{m}
\newcommand{\ft}{f_T}
\newcommand{\fft}{F_T}
\newcommand{\ftk}{f_{T_k}}
\newcommand{\fftk}{F_{T_k}}
\newcommand{\flimit}{\widehat{f}}
\newcommand{\fflimit}{\widehat{F}}

\renewcommand{\phi}{\varphi}

\newcommand\DN{\newcommand}

\DN\url{\textsf}

\DN\lref[1]{Lemma~\ref{#1}}
\DN\tref[1]{Theorem~\ref{#1}}
\DN\pref[1]{Proposition~\ref{#1}}
\DN\sref[1]{Section~\ref{#1}}
\DN\dref[1]{Definition~\ref{#1}}
\DN\rref[1]{Remark~\ref{#1}} 
\DN\corref[1]{Corollary~\ref{#1}}
\DN\eref[1]{Example~\ref{#1}}

\DN\Ni{\mathbb{N}\cup\{ \infty \} }
\DN\hh{\mathfrak{h}}
\DN\hhr{\mathfrak{h}_{r}}
\DN\hhq{\mathfrak{h}_{q}}
\DN\hhh{\hat{\mathfrak{h}}}
\DN\hhi{\hh _{\infty}}
\DN\hhik{\hhi ^{k}}
\DN\hht{H^{x}}
\DN\hhtt{H}
\DN\QQQ{Q}
\DN\QQQall{\tilde{\QQQ }}
\DN\Cb{C_b(\QQQ )}
\DN\dist{\mathsf{d}}
\DN\limi[1]{\lim_{#1\to\infty}} 	

\numberwithin{equation}{section}

\DN\xpp{\x_{p+1}}
\DN\xpm{\x_{p-1}}
\DN\xp{\x_{p}}
\DN\xpa{x_{p}}

\DN\Lnui{L^{\infty}(\nu )}
\DN\Hj{}
\DN\Hi{1_{\mathcal{H}_{\infty} }}
\DN\jj{j}
\DN\kk{k}
\DN\intQ{\int_{Q}}
\DN\cpir{\circ \pi_q }
\DN\Linu{L^{\infty}(\nu )}
\DN\Lone{L^1}

\begin{document}

\title{Absolute continuity and singularity of Palm measures of the 
Ginibre point process
}


\author{Hirofumi Osada \and Tomoyuki Shirai 
}



\maketitle

\begin{abstract}
We prove a dichotomy between absolute continuity and singularity of 
the Ginibre point process $\mathsf{G}$ and its reduced Palm measures
$\{\mathsf{G}_{\mathbf{x}}, \mathbf{x} \in \mathbb{C}^{\ell}, 
\ell = 0,1,2\dots\}$, namely, reduced Palm measures 
$\G_{\mathbf{x}}$ and $\G_{\mathbf{y}}$ 
for $\mathbf{x} \in \mathbb{C}^{\ell}$
and $\mathbf{y} \in \mathbb{C}^{n}$ are mutually 
absolutely continuous if and only if $\ell = n$; 
they are singular each other 
if and only if $\ell \not= n$. 
Furthermore, we give an explicit expression of the Radon-Nikodym density
 $d\G_{\mathbf{x}}/d \G_{\mathbf{y}}$ 
for $\mathbf{x}, \mathbf{y} \in \mathbb{C}^{\ell}$. 

\bigskip 
\textsf{Ginibre point process:  Palm measure:  absolute continuity:  singularity}
\end{abstract}

\noindent 
{\small \footnote{\noindent 
H. Osada:  Faculty of Mathematics, Kyushu University, Fukuoka, 819-0395,  JAPAN
               \\   \quad \quad     \texttt{osada@math.kyushu-u.ac.jp}             \\
  \quad \quad          T. Shirai : 
      Institute of Mathematics for Industry, Kyushu University, Fukuoka, 819-0395,  JAPAN
               \\
          \quad \quad     \texttt{shirai@imi.kyushu-u.ac.jp}           %
}}

\section{Introduction}\label{s:1}
The Ginibre point process $ \G $ is a probability measure on  
the configuration space over $ \mathbb{C}$ ($ \cong \mathbb{R}^2$), whose 
$ k $-correlation function $ \rho_k$ with respect to 
the complex Gaussian measure $\pi^{-1} e^{-|z|^2}dz $ on $\C$ 
is given by the determinant 
\begin{align}\notag 
\rho_{k}  (z_1,\ldots,z_k) = 
\det [ \K (z_i,z_j) ] _{1\le i,j\le k }
\end{align}
of the exponential kernel $\K : \C \times \C \to \C$ defined by  
\begin{align}\notag 
\K(z,w) = e^{z \wbar}. 
\end{align}
It is known that $ \G $ is translation and rotation invariant. 
Moreover, $ \G $ is the weak limit of the distribution $ \G ^{\nn }$ 
of the eigenvalues of the non-Hermitian Gaussian random matrices, 
called Ginibre ensemble of size $\nn$. 
The labeled density $ m^{\nn}$ of $\G^{\nn}$ 
with respect to the Lebesgue measure on $\C^{\nn}$ is then given by 
\begin{align}\label{:11b}
 m^{\nn } (z_1,\ldots, z_{\nn }) &= \frac{1}{\mathcal{Z}_{\nn } }
\prod_{i,j=1 \atop i < j }^{\nn } |z_i-z_j|^2 \prod_{k=1}^{\nn }
e^{-|z_k|^2}. 
\end{align}
Very intuitively, from \eqref{:11b}, 
one may regard $ \G $ as an equilibrium state 
with logarithmic interaction potential (2-dimensional Coulomb potential) 
$ \Psi (z) = - 2 \log |z| $, and $ \G $ has an informal expression 
\begin{align}\label{:11c}&
\G \sim \frac{1}{\mathcal{Z}_{\infty} }
\prod_{i,j=1 \atop i < j }^{\infty} |z_i-z_j|^2 
\prod_{k=1}^{\infty}  e^{-|z_k|^2} dz_k
.\end{align}
Taking the translation invariance of $ \G $ into account, we have 
another informal expression 
\begin{align}\label{:11cc}&
\G \sim \frac{1}{\mathcal{Z}_{\infty}' }
\prod_{i,j=1 \atop i < j }^{\infty} |z_i-z_j|^2 
\prod_{k=1}^{\infty} dz_k. 
\end{align}
Indeed, as we see in Section \ref{s:10}, if $ \mu $ is a translation invariant 
canonical Gibbs measure with interaction potential $ \Psi (x,y) = \Psi (x-y)$ 
and inverse temperature $ \beta >0 $, then $ \mu $ is informally given by 
\begin{align}\label{:11ccc}&
\mu \sim \frac{1}{\mathcal{Z}} \prod_{i,j=1\atop i<j}^{\infty}
e^{-\beta \sum_{i<j, i,j=1}^{\infty} 
\Psi (z_i-z_j)} \prod_{k=1}^{\infty} dz_k
.\end{align}
Taking $ \Psi (z) = -\log |z|$ and $ \beta = 2 $, we obtain \eqref{:11cc}. 
We thus see that the informal expression \eqref{:11cc} is an analogy of that of translation invariant canonical Gibbs measures. 
 
In both cases, we have no straightforward justification because of
the unboundedness of the logarithmic potential at infinity and the presence of
the infinite product of the Lebesgue measure.

When the interaction potential $ \Psi $ is of Ruelle's class, then the associated equilibrium state 
described by the Dobrushin-Lanford-Ruelle equation (DLR equation), and 
the expression corresponding to \eqref{:11c} are surely justified. 
The DLR equations also guarantee the existence of the local density 
in bounded domains 
for fixed outside configuration, and play an important role 
not only for the static problem but also for the dynamical problem 
of the associated infinite particle system. 
However, since $ 2 \log |x| $ is unbounded at infinity, 
one can no longer use the well developed theory based on the DLR equations 
for the Ginibre point process and 
other point processes appearing in the random matrix theory. 

In \cite{o.isde} and \cite{o.rm}, the first author introduced the
notions of quasi-Gibbs property and the logarithmic derivative of $ \G $ 
(see \eqref{:21kk}) to remedy such a situation.  
The quasi-Gibbs property provides a local
density for fixed outside configuration, and the logarithmic derivative gives a
precise correspondence between point processes and potentials.
From these he has deduced that 
the natural labeled stochastic dynamics associated with the Ginibre
point process is given by the infinite-dimensional stochastic differential equation (ISDE): 
\begin{align}\label{:11e}&
dZ_t^i = dB_t^i - Z_t^i dt 
 + \lim_{r\to\infty} 
\sum_{ | Z_t^j | < r  \atop j\not=i , \ j \in\mathbb{N} }
 \frac{ Z_t^i-Z_t^j }{ | Z_t^i-Z_t^j |^2} dt 
\quad (i\in\mathbb{N}) 
.\end{align}
In fact, the Ginibre point process $ \G$ is the equilibrium state of the unlabeled dynamics $ \mathsf{Z}_t = \sum_{i\in\mathbb{N}} \delta_{Z_t^i}$ associated with \eqref{:11e}. 
A surprising feature of this dynamics is that it satisfies 
the second ISDE \cite{o.isde}: 
\begin{align}\label{:11d}&
dZ_t^i = dB_t^i + \lim_{r\to\infty} 
\sum_{ | Z_t^i-Z_t^j | < r  \atop j\not=i  , \ j \in\mathbb{N}  }
 \frac{ Z_t^i-Z_t^j }{ | Z_t^i-Z_t^j |^2} dt 
\quad (i\in\mathbb{N}) 
.\end{align}
The ISDEs \eqref{:11e} and \eqref{:11d} correspond to the informal expressions 
\eqref{:11c} and \eqref{:11cc}, respectively.  

One of the key ingredients of the proof of this result is the small
fluctuation property  
\begin{align}\label{:11h}&
\mathrm{Var}^{\G }(\bra \xis, \ind_{D_r} \ket)  
\sim O(r) \quad (r\to\infty )
\end{align}
obtained by the second author \cite{ST3}. 
Here $ \xis=\sum _{i}\delta _{s_i}$, $\ind_{D_r}$ is the indicator
function of the disk $ D_r =\{ |z| \le r \} $, and 
$ \langle \xis, f \rangle = \sum_{i}f (s_i)$. 
From \eqref{:11h} we see that the order of the variance of 
the Ginibre point process 
is half of that of the Poisson point process with intensity $ dx $. 
This is a result of the strength of 
the 2-dimensional Coulomb interaction at infinity. 

Two typical translation invariant point processes over $ \mathbb{R}^{d} $ 
are the Poisson point process with Lebesgue intensity and 
(randomly shifted) periodic point process. 
The former is the most random point process, 
while the latter is the most deterministic one. 
Translation invariant canonical Gibbs measures 
are the standard class of the point processes 
belonging to the Poisson category. We remark that 
the Ginibre point process has intermediate properties between 
the Poisson and the periodic point processes. 
Indeed, the exponent in \eqref{:11h} is the same as the periodic point
processes and 
the existence of the local density and the associated stochastic dynamics 
like \eqref{:11d} implies that the Ginibre point process 
is similar to the Poisson point processes. 

The most prime interaction potential in $ \mathbb{R}^{d}$ is the 
$ d $-dimensional Coulomb potential although this is outside of the
classical theory of Gibbs measures based on DLR equations. 
In Section \ref{s:10}, 
we introduce the notion of strict Coulomb point processes, which are point processes on $ \mathbb{R}^d$ interacting through $ d $-dimensional Coulomb potential $ \Psi _{d}$ defined as \eqref{:a1} with $\gamma = d $. 
We rigorously formulate the notion of strict Coulomb point processes 
based on the notion of logarithmic derivative of point processes (see \eqref{:21kk} for the definition of the logarithmic derivative). 
As we have seen above, 
the Ginibre point process is the case of the 2-dimensional
Coulomb potential in $ \mathbb{R}^2$ at the inverse temperature $\beta = 2 $, and an example of translation invariant strict Coulomb point processes. 
To our knowledge, the Ginibre point process is 
the only known example of such a type of point processes. 
%
We thus see that the Ginibre point process is one
of the most important interacting particle systems, significantly different from
Poisson and periodic point processes.  Our aim is to study specific
features of the Ginibre point process 
arising from the Coulomb potential and shed new light on it.

In the present paper, we focus on the absolute continuity and the
singularity of the Ginibre point process $\G $ and 
its reduced Palm measures 
$\{\G_{\x} , \x \in \C^{\mm}, \mm=0,1,2\dots\}$ (see \eqref{:21b}). 
Here we interpret $ \G_{\x} = \G $ if $ \mm = 0 $. 
Throughout the paper, by Palm measures we always mean \textit{reduced} 
Palm measures. 

The main results are the following. 

\begin{thm}\label{main1}
Assume that $ \x \in \C^{\mm}$ and $ \y \in \C^{\nnn}$.  
If $ \mm = \nnn$, then 
$\G_{\x}$ and $\G_{\y}$ are mutually absolutely continuous. 
In addition, if $ \mm \not= \nnn$, 
then $\G_{\x}$ and $\G_{\y}$ are singular each other. 
\end{thm}

Such a singularity result of Palm measures is quite different from that
of Gibbs measures. Indeed, if $ \mu $ is a translation invariant
canonical Gibbs measures with Ruelle's class potential, then $
\mu_{\x} \prec \mu $ for any $
\x=(x_1,\ldots,x_{\mm})$, where $ \mu \prec \nu $ means $
\mu $ is absolutely continuous with respect to $ \nu $. 
Roughly speaking, $-\log \frac{d\mu_{\x}}{d\mu}(\xis)$ describes 
the total energy at $\x$ from a given configuration. 
In this sense, informally, $-\log \frac{d\G_{\x}}{d\G}(\xis) = \pm \infty$ 
when $\mm \ne 0$. 

When $\mm = \nnn$, we have an explicit expression of the Radon-Nikodym
density ${d\G_{\x }}/{d\G_{\y}}$.  

\begin{thm}\label{l:21}
For each $\x ,\y \in \C^{\mm}$, the Radon-Nikodym 
density ${d\G_{\x }}/{d\G_{\y }}$ is given by 
\begin{align}\label{:21c}&
\frac{d\G_{\x }}{d\G_{\y }}(\xis) = \ZZ_{\x \y }^{-1}
\lim_{r\to \infty} \prod_{|s_i| < b_r }\frac{|\x -s_i|^2}{|\y -s_i|^2}  
\end{align}
for $\G_{\y }$-a.s. $\xis$. 
Here $\xis=\sum_i\delta_{s_i}$ and $ \{ b_r \}_{r\in \n} $ is an 
increasing sequence of natural numbers.  
We use a convention such that
\begin{align}\notag 
|\x -s| = \prod_{j=1}^{\mm}|x_j - s| 
\quad \text{for $ \x =(x_1,\ldots, x_{\mm})$ and $s \in \C$}.  
\end{align}
The convergence in \eqref{:21c} takes place 
uniformly in $\x$ on any compact set in $\C^{\mm}$. 
Moreover, the normalization constant $\ZZ_{\x \y }$ is given by 
\begin{align}\label{:21g}& 
 \ZZ_{\x \y } = Z(\x) Z(\y)^{-1}. 
\end{align}
Here $Z(\x)$ is the smooth function on $\C^{\mm}$ defined 
as the unique continuous extension of the function 
\begin{equation}
Z(\x) = \frac{\det [\K(x_i,x_j)]_{i,j=1}^{\mm}}{|\Delta(\x)|^2}
\label{:zx}
\end{equation}
defined for $\x \in \C^{\mm}$ with $\Delta(\x) \not= 0$, where 
$\Delta(\x) = \prod_{1 \le i < j \le \mm} (x_i - x_j)$ 
when $\mm \ge 2$; $\Delta(\x) = 1$ when $\mm=1$.   
\end{thm}

For example, by \lref{Phiz}, we see that 
\[
Z({\bf 0}) = \left(\prod_{k=1}^{\mm-1} k!\right)^{-1}. 
\]
In general, $Z(\x)$ can be expressed as a series expansion by Schur functions. 
See \rref{rem:zx} in Section~9. 

To prove the singularity between 
$\G_{\x} $ and $\G_{\y} $ when $\mm \not= \nnn$, we introduce 
a family of real-valued functions $\{\fft\}_{T > 0} $ 
defined on the configuration space $Q$ over $\C$ by 
\begin{align}\label{:13a}&
\fft(\xis) = \frac{1}{T} \int_{0}^{T} (\xis(D_{\sqrt{r}}) - r) dr, 
\end{align}
where $D_r$ is the disk of radius $r$ and $\xis(A)$ is the number of
points inside a measurable set $A$. 
\begin{thm}\label{l:13}
Let $\G_{\x}$ be the Palm measure of $\G$ conditioned at  
$\x \in \C^{\mm}$. Then 
the function $ \fft$ converges weakly in $ L^2(Q, \G_{\x} )$ and
 satisfies  
\begin{align}\notag
\lim_{T \to \infty} \fft = - \mm. 
\end{align}
\end{thm}

Suppose that the number $\mm$ of the conditioned particles is unknown. 
Then Theorem~\ref{l:13} implies one can detect the value $ \mm $ 
from the sample point $\xis$ almost surely. 
Namely, the total system of the sample point $\xis$ memorizes 
the missing number $ \mm $. 
This property is a remarkable contrast to the Poisson 
and canonical Gibbs measures, 
and implies that 
the Ginibre point process is similar to the periodic 
point process from this view point. 
We may thus regard the Ginibre point process as a random crystal. 
Similar phenomenon is also observed as rigidity in Ginibre point process
in \cite{Ghosh1,Ghosh2}, where 
it is shown that the conditioning of configuration outside of a disk 
determines the number of particles inside the disk. 
In \cite{HS} the problem of absolutely continuity of 
point processes is discussed in terms of deletion-insertion tolerance, 
and it is shown that the Gaussian zero process on the plane $\C$ is 
neither insertion tolerant nor deletion tolerant and 
that on the hyperbolic plane $\mathbb{D}=\{ |z|<1 \} $ is 
both insertion tolerant and deletion tolerant. 

The organization of this paper is as follows. 
In Section 2, we give the setting of this paper. 
In Section 3, we recall the notion of Palm measures and 
give remarks on Palm measures of determinantal point processes. 
In Section 4, we provide basic properties of Ginibre point process. 
In Section 5, we express the kernel of a continuous version of 
Palm measures in terms of Schur functions, and give an estimate of 
the difference between a correlation kernel and its Palm kernels. 
Section 6 deals with uniform estimate of variances for Ginibre and its
Palm measures by using the results in Section 5. 
In Section 7, we give proofs for 
the latter half of \tref{main1} and \tref{l:13}, which show singularity between 
the Ginibre point process and its palm measures. 
In Section 8, we give a sufficient condition for 
absolute continuity between two point processes 
that have an approximation sequence of finite point processes, 
and by applying it to our problem 
we prove the former half of \tref{main1} and \tref{l:21} in Section 9. 
In Section 10, we give concluding remarks with some open questions. 

\section{Setup} \label{s:2}

Let $R$ be a Polish space. A configuration $\xis$ in $R$ is 
a Radon measure of the form $\xis = \sum_i \delta_{s_i}$. 
Here $\delta_a$ denotes 
the delta measure at $a$ and $\{s_i\}$ is a countable sequence in $R$ 
such that $ \xis (K) < \infty $ for all compact set $ K $ in $R$. 
For a measurable function $f$, we denote 
$\bra \xis, f \ket = \int_R f(s) \xis(ds) = \sum_{i} f(s_i)$ 
whenever the right-hand side makes sense. 

We regard the zero measure as 
an empty configuration by convention, which describes the state that 
no particle exists. We remark that $\xis(A)$ becomes the number of
particles in a measurable set $A$. 
Let $Q=Q(R)$ be the set consisting of all such
configurations on $R$. We endow $Q$ with the vague topology, under which 
$Q$ is again a Polish space. 
A $Q$-valued random variable $\xis = \xis(\omega)$ is called 
a \textit{point process} or a \textit{random point field}. 
In what follows, we also refer to 
its probability distribution $\mu$ on $Q$ as a point process.  

We fix a Radon measure $\radon$ on $(R, \bcal(R))$ as a reference measure. 
We call a symmetric measure $\la_n$ 
on $R^n$ the $n$-th correlation measure if it satisfies 
\begin{align*}
\E\left[\prod_{i=1}^j {\xis(A_i)! \choose k_i!} k_i!
\right] 
&= \int_Q \prod_{i=1}^j \frac{\xis(A_i)!}{(\xis(A_i) - k_i)!} \mu(d\xis) \\
&= \la_n(A_1^{k_1} \times \cdots \times A_j^{k_j}).  
\end{align*}
Here $A_1, \dots, A_j \in \bcal(R)$  are disjoint and $k_1,\dots, k_j \in
\n$ such that $k_1+\cdots + k_j = n$. 
If $\xis(A_i) - k_i \le 0$, we interpret 
${\xis(A_i)!}/{(\xis(A_i) -k_i)!}=0$. 
Furthermore, if $\la_n$ is absolutely continuous with respect to 
$\radon^{\otimes n}$, the Radon-Nikodym derivative 
$\rho_n(x_1,\dots, x_n)$ is said to be 
the $n$-point correlation function with respect to $\radon$, i.e., 
\[
 \la_n(dx_1 \dots dx_n)
= \rho_n(x_1, \dots, x_n) \radon^{\otimes n}(dx_1 \dots dx_n). 
\] 

Let $K : R \times R \to \C$ be a Hermitian symmetric kernel such that 
the associated 
integral operator $(Kf)(x) := \int_R K(x,y)f(y)\radon(dy)$ becomes a locally
trace class operator on $L^2(R,\radon)$ 
satisfying $\text{Spec}(K) \subset [0,1]$. 
In other words, $0 \le  (Kf,f)_{L^2(R, \radon)}\le
(f,f)_{L^2(R,\radon)}$ for any $f \in L^2(R,\radon)$. 
It is known \cite{So,ST} that 
under these conditions on a pair $(K, \radon)$, 
there exists a unique point process $\mu = \mu_{K,\radon}$
such that its $n$-point correlation function
$\rho_n$ with respect to $\radon$ is given by 
\begin{equation}\notag
\rho_n(x_1,\dots, x_n) = \det[K(x_i,x_j)]_{1 \le i,j \le n}
\end{equation} 
for every $n \in \n$. 
We call it the {\it determinantal point process} (DPP) 
associated with $(K,\radon)$. 
We note that if we set
\begin{equation}
 \tilde{K}(x,y) = g(x)^{1/2} K(x,y) g(y)^{1/2}, \quad 
\tilde{\radon}(dx) = g(x)^{-1} \radon(dx)
\label{:gauge}
\end{equation}
for $g : R \to (0, \infty)$, 
then $(\tilde{K}, \tilde{\radon})$ defines the same DPP 
as that of $(K,\radon)$. 

In the rest of the paper we will take $R=\C$, and $Q$ denotes the
configuration space over $\C$. Let $\gsf(z) = \pi^{-1} e^{-|z|^2}$ and 
$\radon(dz) = \gsf(dz) := \gsf(z)dz$ be the Gaussian measure on $\C$. 
Let $\K : \C \times \C \to \C$ such that 
\begin{equation}\notag
 \K(z,w) = e^{z \wbar}. 
\end{equation}
Let $\G$ be the probability measure on $Q$ 
whose $n$-point correlation function $ \rho_n$ 
with respect to the Gaussian measure $\gsf$ is given by 
\begin{align}\notag
\rho_n (z_1,\ldots, z_n)= \det [\K (z_i,z_j)]_{1\le i,j \le n}
.\end{align}
The probability measure $\G$ is 
a DPP associated with $(\K, \gsf(dz))$ and 
called the Ginibre point process.  
$\G$ will denote the Ginibre point process in the rest of the paper. 

Let $\mm \in \n \cup \{0\}$.  
For $\x \in \C^{\mm}$ and 
a probability measure $\mu$ on $Q$, 
let $\mu_{\x }$ be the reduced Palm measure of $ \mu $ conditioned at
$\x  = (x_1,\ldots,x_{\mm })$: 
\begin{align}\label{:21b}&
\mu _{\x } = \mu (\cdot - \sum_{i=1}^{\mm} \delta _{x_i} | 
\xis(\{x_i\})\ge 1 \text{ for all }i). 
\end{align}
As we will see later, one can take the continuous version of $\mu_{\x}$ 
for DPPs in such a way that $\mu_{\x}$ 
is weakly continuous in $\x$. So $\mu_{\x}$ is defined for all $\x \in
\C^{\mm}$ with no ambiguity. 
If $\mm=0$, then we understand $\mu_{\x} = \mu$ 
for $\x \in \C^{0}$ by convention. 

The relation between the Ginibre point process 
and the two-dimensional Coulomb
potential with the inverse temperature $ \beta = 2 $ has been rigorously
established in \cite{o.isde} by using the notion of logarithmic derivative.

Let $ \mu ^1 $ be the 1-Campbell measure of a point process 
$ \mu $. By definition 
$ \mu ^1 $ is a measure on $ \mathbb{C}\times Q $ given by 
\begin{align}\notag
\mu ^1 (dxd\xis ) = \la_1 (dx) \mu _{x}(d\xis ) 
= \rho_1 (x) \radon(dx) \mu _{x}(d\xis ) 
,\end{align}
where $ \rho _1$ is the 1-correlation function of $ \mu $ 
with respect to the Lebesgue measure and $ \mu _{x} $ is 
the Palm measure conditioned at  $ x \in \mathbb{C}$. 
We call a function 
$ \mathsf{d}^{\mu } \in L_{\mathrm{loc}}^1(\mathbb{C}\times Q ,\mu^{1})$ 
a logarithmic derivative of $ \mu $ if 
\begin{align}\label{:21kk}&
- \int_{\mathbb{C}\times Q } \nabla _{x}  \varphi (x, \xis) \mu ^1 (dxd\xis ) = 
\int _{\mathbb{C}\times Q } 
\varphi (x,\xis ) \mathsf{d}^{\mu } (x,\xis ) 
\mu ^1 (dxd\xis )
\end{align}
for all  $ \varphi \in C_{0}^{\infty}(\mathbb{C})\otimes C_b(Q ) $,
where $C_b(Q)$ is the space of all bounded continuous functions on $Q$. 
It was proved in \cite[Th.\ 61]{o.isde} that 
\begin{align}\label{:21l}&
 \mathsf{d}^{\G} (x,\xis ) = \lim_{r\to\infty }
 \sum_{|s_i -x| < r } \frac{2(x-s_i)}{|x-s_i|^2} \quad \text{ in }
 L^2_{\mathrm{loc}}(\mathbb{C}\times Q , \G^1 ).
\end{align}
One may regard the conditions \eqref{:21kk} and \eqref{:21l} 
as a differential type of the DLR equation for Ginibre point process. 
We emphasize that this formulation is valid for more general setting. 
See the discussion in Section 10.

Note that there seems to exist no drift term coming from 
free potential in \eqref{:21l}. 
However, in the finite particle approximation of Ginibre ensemble, 
there exists an Ornstein-Uhlenbeck type center force $-2x$. 
We remark that the convergence in \eqref{:21l} is conditional convergence. 
This follows from the fact that the effect of the Coulomb interactions
at infinity is quite strong. 
In addition, we have a second representation of $ \mathsf{d}^{\G}$ 
\cite[(2.11)]{o.isde}: 
\begin{align}\notag
\mathsf{d}^{\G} (x,\xis ) = - 2x + \lim_{r\to\infty }
 \sum_{|s_i| < r } \frac{2(x-s_i)}{|x-s_i|^2} \quad \text{ in }
 L^2_{\mathrm{loc}}(\mathbb{C}\times Q, \G^1 )
,\end{align}
which one may expect from the $ \nn $-particle approximation \eqref{:11b} 
to the Ginibre point process $\G$. 
The similar phenomenon is also 
pointed out in \cite{CPPR} for Poisson point processes on
$\mathbb{R}^d$ with $d \ge 3$.

\section{Palm measures of determinantal point processes}

Recall that we always mean by Palm measures 
\textit{reduced} Palm measures.  

Let us restate the definition of Palm measures for simple point
processes. 
For a simple point process $\mu$ on $R$, the following formula defines 
Palm measures $\{\mu_x, x \in R\}$ for $\la_1$-a.e.$x$:  
for any bounded measurable function $F$ on $R \times Q$
\[
\int_Q \mu(d\xis) \int_R \xis(dx) F(x, \xis) 
= \int_R \la_1(dx) \int_Q \mu_x(d\xis) F(x, \xis + \delta_x),  
\]
where $\la_1$ is the $1$-correlation measure. 
If $\int_{\xis(U) \ge 2} \mu(d\xis) \xis(U) = o(\la_1(U))$ 
as $U \searrow \{x\}$, then 
\[
\lim_{U \searrow \{x\}} 
E[f \ | \ \xis(U) \ge 1] = \int_Q f(\xis) \mu_x(d\xis) \quad 
\la_1\text{-a.e.} x
\]
for any bounded measurable function $f$ on $Q$. 
Similarly, the formula 
\[
\int_Q \mu(d\xis) \int_{R^n} \xis^{\otimes n}(d\x) F(\x, \xis) 
= \int_{R^n} \la^{(n)}(d\x) \int_Q \mu_{\x}(d\xis) F(\x, \xis + 
\sum_{i=1}^n \delta_{x_i})
\]
for any bounded measurable function $F$ on $R^n \times Q$
defines Palm measures $\{\mu_{\x}, \x \in R^n\}$ 
for $\la^{(n)}$-a.e.$\x=(x_1,\dots,x_n)$, 
where $\la^{(n)}$ is the $n$-th moment measure. 
For distinct $\x = (x_1,\dots, x_n) \in R^n$, one can think that 
$\mu_{\x}$ is defined for $\la_n$-a.e.$\x$. 
It is well-known that Poisson point process $\Pi_{\nu}$ with 
intensity measure $\nu$ do not change 
by the operation of taking Palm measure, that is,   
$(\Pi_{\nu})_{\x} = \Pi_{\nu}$ for any $\x \in R^n$. 

The next fact shows that the Palm measures of a DPP are again DPPs. 
\begin{prop}[\cite{ST}] \label{l:31}
Let $\mu_K$ be a DPP associated with kernel $K(x,y)$. 
Set 
\begin{equation}
 K_b(x,y) = K(x,y) - \frac{K(x,b)K(b,y)}{K(b,b)}
\label{:31a}
\end{equation}
for $b \in R$ with $K(b,b)>0$. Then, for $\la_1$-a.e. $b \in R$, we have 
\begin{align}\notag
(\mu_K)_b = \mu_{K_b}
.\end{align}
Furthermore, for $\bb = (b_1,\dots, b_n) \in R^n$ 
such that $\det [K(b_i,b_j)]_{i,j=1}^n > 0$, 
set
\begin{align}\label{:31c}&
 K_{\bb}(x,y) = 
\frac{\det [K(p_i,q_j)]_{i,j=0}^n}{\det [K(b_i,b_j)]_{i,j=1}^n}, 
\end{align}
where $p_0=x, q_0 = y$ and $p_i = q_i = b_i$ for $i=1,2,\dots, n$. 
Then, for $\la_n$-a.e. $\bb \in R^n$, we have 
\begin{align}\notag
(\mu_K)_{\bb} = \mu_{K_{\bb}}
.\end{align}
\end{prop}

As we have seen in \pref{l:31}, 
Palm measures $\{\mu_{\bb}\}$ are defined for $\la_{n}$-a.e.
$\bb \in R^n$, i.e., for $\bb = (b_1,\dots,b_n)$ 
such that $\det [K(b_i, b_j)]_{i,j=1}^n >0$. 
The next lemma shows that if the kernel $K$ is smooth one can take 
a continuous version of Palm measures and define $\mu_{\bb}$ 
even when $\det [K(b_i, b_j)]_{i,j=1}^n =0$.

For later use, we only consider the case where DPPs on $\C$
associated with an analytic, Hermitian kernel $K$. 
Other cases can also be discussed in the same manner. 
\begin{lem}\label{conti-version} 
Let $\mu = \mu_{K}$ be a DPP on $\C$ generated by a Hermitian 
kernel $K : \C \times \C \to \C$ which is analytic in the first variable
 (and thus anti-analytic in the second) 
and a Radon measure $\radon$ with positive smooth
 density. Then, we can take a continuous version of Palm measures 
$\{\mu_{\x}\}_{\x \in \C^n}$ 
even when $\det (K(x_i, x_j))_{i,j=1}^n$ is 
vanishing at some points $\x=(x_1,\dots,x_n)$ in $\C^n$. 
\end{lem}
\begin{proof}
For simplicity, we only show the case where $n=1$. 
Suppose that $K(b, b)=0$. Since $K$ is analytic and Hermitian, 
there exist a $p\in \n$ and an analytic kernel $L$  
with $L(b, b) > 0$ such that 
$K(z,w) = L(z,w) (z-b)^p \overline{(w-b)^p}$. 
It is easy to see that the Palm kernel $K_x$ 
for $\mu_x \ (x \in \C)$ is of the form 
$K_x(z,w) = L_x(z,w) (z-b)^p \overline{(w-b)^p}$ when $x \not= b$. 
Hence, as $x \to b$, DPP $\mu_x$ converges weakly to DPP $\mu_b$ 
with kernel $K_b(z,w):=L_b(z,w) (z-b)^p \overline{(w-b)^p}$. 
\qed\end{proof}

\begin{ex}\label{version}
For the Ginibre point process $\G$ with kernel $\K(z, w) = e^{z \wbar}$, 
by \pref{l:31}, 
we can take a continuous version of Palm
 measures $\{\G_{\alpha}, \alpha \in \C\}$ 
since $\K(\alpha, \alpha) = e^{|\alpha|^2} > 0$. 
In particular, $\G_0$ is the DPP with kernel 
$\K_0(z,w) = e^{z\wbar} -1$. 
Since $\K_0(0,0)=0$, we cannot define the kernel $(\K_0)_0$ 
by using \eqref{:31a}. 
However, by \lref{conti-version}, 
we can take a continuous version $\{\G_{\x}, \x \in \C^2\}$. 
The kernel $\K_{\x}$ is defined by \eqref{:31c} when $\x \in \C^2$ is not a
 diagonal element and by \lref{conti-version} otherwise; in particular,  
\begin{align}\notag
\K_{(0,0)}(z,w) = e^{z \wbar} - 1 - z \wbar.  
\end{align}
Repeating this procedure, we can easily see that 
for $\oo_{\mm} =(0,0,\dots,0) \in \C^{\mm}$ 
the Palm measure $\G_{\oo_{\mm}}$ can be taken as 
the DPP associated with kernel 
\begin{align}\label{:31f}&
 \K_{\oo_{\mm}}(z,w)  = \sum_{k=\mm}^{\infty}
 \frac{(z\wbar)^{k}}{k!}. 
\end{align}
Palm kernel $\K_{\x}$ conditioned at $\x$ 
will be given in terms of Schur functions in \lref{l:91}. 
\end{ex}

\begin{rem}
(1) For DPPs, the diagonal $K(z,z) (= \rho_1(z))$ is the density 
(w.r.t. $\radon$) of points at $z$. 
It is obvious from \eqref{:31a} that $K_x(z,z) \le K(z,z)$, which
 implies that taking Palm measure decreases the density. 
By induction, we see that, for any $\x$ and $z \in R$, we have 
\begin{equation}
 K_{\x}(z,z) \le K(z,z), 
\label{decrease}
\end{equation}
i.e., $\rho_{1,\x}(z) \le \rho_1(z)$. \\
(2) The trivial inequality $\rho_2(z,w) \ge 0$ for DPPs 
implies the Schwarz inequality 
\begin{equation}
 |K(z,w)|^2 \le K(z,z) K(w,w). 
\label{schwarz} 
\end{equation}
\end{rem}

The following variance formulas are useful 
for deriving small fluctuation properties of $\G$ and 
its Palm measures. 

\begin{lem}\label{l:35}
Let $\mu_{K,\radon}$ be a DPP associated with $(K,\radon)$. 
Then, 
\begin{align}
 \var^{\mu_{K,\radon}}(\bra \xis, g\ket) 
&= \int_R |g(z)|^2 K(z,z) \radon(dz) - \int_{R^2} g(z) \overline{g(w)}
 |K(z,w)|^2 \radon(dz)\radon(dw),  
\label{var0} 
\end{align}
where $\bra \xis, g \ket = \int_R g(x)\xis(dx)$. 
Moreover, suppose that the kernel $K$ has reproducing property, i.e., 
\begin{equation}
\int_R K(z,u) K(u,w) \radon(du) = K(z,w). 
\label{reproducing} 
\end{equation}
Then, the following also holds: 
\begin{align}
 \var^{\mu_{K,\radon}}(\bra \xis, g\ket) 
&= \frac{1}{2} \int_{R^2} |g(z) - g(w)|^2 |K(z,w)|^2 \radon(dz) \radon(dw). 
\label{var} 
\end{align}
\end{lem}
\begin{proof}
It is easy to see that \eqref{var0} and \eqref{var} hold. 
(see cf. page 195 \cite{OS}.)
\qed\end{proof}

Reproducing property is preserved by the operation of 
taking Palm measures. 
\begin{lem}\label{l:reproducing} 
Let $K$ be a kernel with reproducing property (\ref{reproducing}). 
Then, so is $K_{\x}$ for $\la_n$-a.e. $\x \in R^n$. 
\end{lem}
\begin{proof}
It suffices to show it for $n=1$. 
 \eqref{:31a} and \eqref{reproducing} yield the reproducing property 
for $K_x$. 
\qed\end{proof}

\section{Ginibre point process and its basic property}\label{s:4}
In this section, we summarize the basic properties of 
Ginibre point process. 

\begin{prop}[\cite{G}]\label{ginibre}
Let $A_{\nn}$ be an $\nn \times \nn$ matrix with i.i.d. standard complex 
Gaussian entries, i.e., $(A_{\nn})_{ij} \sim N_{\C}(0,1)$. 
Then the joint probability density of $\nn$ eigenvalues are given by 
\begin{equation}
 p^{\nn}(\s) = \frac{1}{\prod_{k=1}^{\nn} k!} 
\prod_{1 \le i<j \le \nn} |s_i - s_j|^2 
\notag\label{:40a}
\end{equation}
 for $\s=(s_1,s_2,\dots,s_{\nn}) \in \C^{\nn}$ 
with respect to the Gaussian measure 
$\gggn (d\s) = \prod_{i=1}^{\nn} \gsf(ds_i)$. 
Equivalently, the set of eigenvalues of $A_{\nn}$ forms 
the DPP $\G^{\nn}$ associated with $\gsf(dz)$ and 
\begin{align}\label{:40b}&
 \K^{\nn}(z,w) = \sum_{k=0}^{\nn-1} \frac{(z \wbar)^{k}}{k!}
.\end{align}
\end{prop}

Since $\K^{\nn}(z,w) \to \K(z,w)$ uniformly on any compact set in 
$\C \times \C $, the DPP $\G^{\nn}$ associated with $(\K^{\nn}, \gsf)$ 
converges weakly to the DPP $\G$ associated with $(\K, \gsf)$ 
(see e.g. Proposition~3.1 in \cite{ST}). 

\begin{prop}
Let $\x = (x_1, \dots, x_{\mm}) \in \C^{\mm}$ and 
\[
\ZZ(\G_{\x}^{\nn}) = 
\int_{\C^{\nn}}\prod_{i=1}^{\nn} |\x - s_i|^2 
|\Delta(\s)|^2 \gggn (d\s), 
\]
where $\Delta(\s)$ is the Vandermonde determinant as in \tref{l:21}. 
Then we have 
\begin{align}
\lim_{\nn \to \infty} 
\frac{\ZZ(\G_{\x}^{\nn})}{\ZZ(\G_{\y}^{\nn})} 
= \left|\frac{\Delta(\y)}{\Delta(\x)}\right|^2 \cdot 
\frac{\det[\K(x_i,x_j)]_{i,j=1}^{\mm}}{\det [\K(y_i,y_j)]_{i,j=1}^{\mm}} 
\label{:42c}&
\end{align}
for $\x, \y \in \C^{\mm}$ with $\Delta(\x)\not= 0$ and $\Delta(\y) \not= 0$. 
\end{prop}
\begin{proof}
By Proposition~\ref{ginibre}, 
for $\x \in \C^{\mm}$ and $\s \in \C^{\nn}$, 
\begin{equation}
p^{\mm+\nn}(\x, \s)
= \frac{1}{\prod_{k=1}^{\mm+\nn} k!} 
|\Delta(\x)|^2
\prod_{i=1}^{\nn} |\x - s_i|^2 
|\Delta(\s)|^2
\label{joint}
\end{equation}
is the joint probability density function for eigenvalues of 
Ginibre random matrix of 
size $\mm+\nn$ with respect to the Gaussian measure 
$\gggm (d\x) \gggn (d\s)$. 
Then, the $\mm$-point correlation function of $\G^{\mm+\nn}$ is given by 
\begin{align*}
\rho_{\mm}^{(\mm+\nn)}(\x) 
&
= \frac{(\mm+\nn)!}{\mm!} \int_{\C^{\nn}} 
p^{\mm+\nn}(\x, \s) \gggn (d\s)\\
&= \frac{(\mm+\nn)!}{\mm!} 
\frac{|\Delta(\x)|^2}{\prod_{k=1}^{\mm+\nn} k!} 
\int_{\C^{\nn}} \prod_{i=1}^{\nn} |\x - s_i|^2 
|\Delta(\s)|^2 \gggn (d\s) \\
&= \frac{(\mm+\nn)!}{\mm!} \frac{|\Delta(\x)|^2}{
\prod_{k=1}^{\mm+\nn} k!} \ZZ(\G_{\x}^{\nn}) 
.\end{align*}
On the other hand, since $\G^{\mm+\nn}$ is the DPP associated with
 $\K^{\mm+\nn}$, we have 
\[
 \rho_{\mm}^{(\mm+\nn)}(\x) 
= \det[\K^{\mm+\nn}(x_i,x_j)]_{i,j=1}^{\mm},  
\]
where $\K^{\mm+\nn}$ is as in \eqref{:40b}. 
Hence, we obtain 
\begin{equation}
\frac{\ZZ(\G_{\x}^{\nn})}{\ZZ(\G_{\y}^{\nn})}
= \left|\frac{\Delta(\y)}{\Delta(\x)}\right|^2
\frac{ \det[
 \K^{\mm+\nn}(x_i,x_j)]_{i,j=1}^{\mm}}{
\det[\K^{\mm+\nn}(y_i,y_j)]_{i,j=1}^{\mm}}
\to 
\left|\frac{\Delta(\y)}{\Delta(\x)}\right|^2
\frac{\det[\K(x_i,x_j)]_{i,j=1}^{\mm}}{\det[ \K(y_i,y_j)]_{i,j=1}^{\mm}} 
\label{:42e}
\end{equation}
as $\nn \to \infty$. 
\qed\end{proof}
When $\Delta(\x) =0$, for example, $\x = \oo_{\mm}$, 
\eqref{:42c} should be understood by using the following 
(see \sref{s:50} for more general cases). 
\begin{lem}\label{Phiz}
Let $\x = (x_1, \dots, x_{\mm}) \in \C^{\mm}$. 
For every $\nn \in \n$, 
\[
\lim_{\x \to \oo_{\mm}} \frac{
\det[
 \K^{\mm+\nn}(x_i,x_j)]_{i,j=1}^{\mm}}{|\Delta(\x)|^2}
= \lim_{\x \to \oo_{\mm}} \frac{
\det[
 \K(x_i,x_j)]_{i,j=1}^{\mm}}{|\Delta(\x)|^2}
= \prod_{k=0}^{\mm-1} \frac{1}{k!}. 
\]
\end{lem}
\begin{proof}
We consider the $\mm \times \qq$ matrix $V^{\qq}(\x) 
= (\phi_p(x_i))_{1 \le i \le \mm, 0 \le p \le \qq-1}$ with $\qq \ge
\mm$, where $\phi_p(z) = z^p/\sqrt{p!}$.  
Then, 
\begin{equation}
[\K^{\mm+\nn}(x_i,x_j)]_{i,j=1}^{\mm}
= V^{\mm+\nn}(\x)
V^{\mm+\nn}(\x)^*. 
\label{:kmn} 
\end{equation}
By the Binet-Cauchy formula, as $\x \to \oo_{\mm}$, we see that  
\begin{align*}
\det[\K^{\mm+\nn}(x_i,x_j)]_{i,j=1}^{\mm}
&= \det\Big[V^{\mm}(\x) V^{\mm}(\x)^*\Big] 
+ \text{(higher order terms)} \\
&= \left(\prod_{k=0}^{\mm-1} \frac{1}{k!}\right) 
|\Delta(\x)|^2 (1+o(1)). 
\end{align*}
Since $\K^{\mm+\nn}(z,w)$ converges to $\K(z,w)$ as $\nn \to \infty$ 
uniformly on any compact set in $\C \times \C$, 
the second equality also holds. 
\qed\end{proof}

\begin{prop}[\cite{kostlan}]\label{kostlan}
The set of squares of moduli of the Ginibre points is equal in law to 
$\{Y_i, i=1,2,\dots\}$, where 
$\{Y_i\}_{i=1}^{\infty}$ is a sequence of independent random
variables such that each $Y_i$ obeys $\Gamma(i,1)$, i.e., 
\[
 \Prob(Y_i \le t) = \int_0^t \frac{s^{i-1} e^{-s}}{\Gamma(i)} ds \quad (t \ge
 0). 
\]
\end{prop}

This proposition can be generalized to the case where 
radially symmetric DPPs on the plane \cite{HKPV}. 
Let $\radon$ be a rotation invariant finite measure on $\C$ and suppose
that $\{\phi_j(z) = a_j z^j\}_{j=0}^{\infty}$ is an orthonormal system 
with respect to $\radon$. We consider the kernel 
\begin{equation}
 K(z,w) 
= \sum_{j=0}^{\infty}\phi_j(z) \overline{\phi_j(w)}
= \sum_{j=0}^{\infty} |a_j|^2 (z\wbar)^j
\label{eq:kzw}
\end{equation}
and the DPP $\mu_K$ associated with $K$ and $\radon$. 
For simplicity, we assume that $\radon(dz) = g(|z|) dz$ 
for some $g : [0,\infty) \to [0,\infty)$, where 
$dz$ is the Lebesgue measure on $\C$. 
We consider independent non-negative random variables $\{Z_j, j=0,1,2,\dots\}$ 
whose law are given by $ f_j(t)dt $, where 
\begin{align}\label{:44b}&
 f_j(t) = \pi |a_j|^2 t^j g(\sqrt{t}), \quad \text{ for each }j 
.\end{align}
In \lref{gen_kostlan} and \corref{:46a} below, 
we consider the map $ \Theta : Q(\C) \to Q([0,\infty))$ defined by 
\begin{equation}
 \Theta(\sum_i\delta_{s_i}) = \sum_i \delta_{|s_i|^2}.  
\notag\label{theta}
\end{equation}

Although a proof of \lref{gen_kostlan} can be found in \cite{HKPV}, 
here we include a slightly different proof for readers' convenience. 

\begin{lem}[\cite{HKPV}] \label{gen_kostlan}
Let $ \mu_K $ and $ \{Z_j, j=0,1,2,\dots\}$ be as above. 
Then 
\begin{align}\label{:44d}&
\mu_K \circ \Theta^{-1} = 
\text{the law of $\sum_{j=0}^{\infty}\delta_{Z_j}$}.   
\end{align}
\end{lem}
\begin{proof}
Let $A_{r,R} = \{z \in \C ; \sqrt{r} \le |z| \le \sqrt{R}\}$ for $0 \le r < R$. 
Then, each $\phi_j(z), j=0,1,\dots$ is an eigenfunction of 
the restriction operator $K_{A_{r,R}}$ and the corresponding eigenvalue
is given by 
\begin{equation}
\la_j(r,R) 
= \int_{A_{r,R}} |\phi_j(z)|^2 \radon(dz) 
= \int_{r}^{R} f_j(t)dt
= \Prob(Z_j \in [r, R]), 
\label{eigenrR}
\end{equation}
where $f_j(t)$ is defined as in \eqref{:44b}. 
For disjoint annuli $\{A_{r_k, R_k}\}_{k}$ having the origin as common
 center, we consider functions represented as 
$g = \sum_k c_k \ind_{A_{r_k,R_k}}$ and 
$g_0 = \sum_k c_k \ind_{[r_k,R_k]}$. Then, 
since $1-e^{-g} = \sum_{k} (1-e^{-c_k}) \ind_{A_{r_k,R_k}}$ and 
the system $\{\phi_j(z), j=0,1,\dots\}$ are simultaneous eigenfunctions 
of $\{K_{A_{r_k,R_k}}\}$, it follows from \eqref{eigenrR} that 
\begin{align*}
K((1-e^{-g})\cdot \phi_j) 
&= \sum_k (1-e^{-c_k}) K_{A_{r_k,R_k}} \phi_j \\
&= \sum_k (1-e^{-c_k}) \Prob(Z_j \in [r_k, R_k]) \cdot \phi_j\\
&= \E[1-e^{-g_0(Z_j)}] \phi_j
\end{align*}
Hence, the nonzero eigenvalues of the operator 
$K((1-e^{-g})\cdot )$ are given by 
$\{\E[1-e^{-g_0(Z_j)}]\}_{j=0}^{\infty}$. 
Then, we have 
\begin{align}
\int_{Q(\C)} e^{-\bra \xis, g \ket} \mu_K(d\xis) 
&= \det\{I - K(1-e^{-g})\} \notag\\
&= \prod_{j=0}^{\infty} \E[e^{-g_0(Z_j)}]
= \E[e^{-\bra \sum_{j=0}^{\infty} \delta_{Z_j}, g_0 \ket}]
\label{laplace}
\end{align}
The first equality follows from the well-known formula for 
Laplace transform of DPP (see \cite{ST}). 
On the other hand, by the definition of the map $\Theta$, 
we see that 
\begin{equation}
\int_{Q([0,\infty])} e^{-\bra \eta, g_0 \ket} (\mu_K \circ
 \Theta^{-1})(d\eta) 
= \int_{Q(\C)} e^{-\bra \xis, g \ket} \mu_K(d\xis). 
\label{pushforward}
\end{equation}
Therefore, the standard limiting argument together with 
\eqref{laplace} and \eqref{pushforward} 
yield \eqref{:44d}. 
\qed\end{proof}

\begin{cor}\label{:46a}
Let $\{ Y_i \}_{i=1}^{\infty}$ be as in \pref{kostlan} and $ \G_{\oo_{\mm}} $ 
be the Palm measure of $\G$ 
conditioned at $ \oo_{\mm}=(0,\ldots,0) \in \mathbb{C} ^{\mm}$. 
Then,  
\begin{align}\label{:44e}&
\G_{\oo_{\mm}} \circ \Theta^{-1} 
= \text{the law of $\sum_{i=\mm+1}^{\infty}\delta_{Y_i}$}.   
\end{align}
\end{cor}
\begin{proof}
Let $\radon(dz) = \pi^{-1} e^{-|z|^2} dz$, and set 
$a_j = 1/\sqrt{j!}$ for $j \ge \mm$; $0$ otherwise in \eqref{eq:kzw}. 
Then $Z_j = 0$ for $0 \le j \le \mm-1$, and 
the law of $Z_j$ is $\frac{t^j}{j!} e^{-t} dt$ 
for $j \ge \mm$ by \eqref{:44b}, 
which is equal to that of $Y_{j+1}$. 
The corresponding kernel $K$ is 
$\K_{\oo_{\mm}}(z,w)$ from \eqref{:31f}, and then the corresponding 
DPP is $\G_{\oo_{\mm}}$.
By \eqref{:44d}, we have 
\[
\G_{\oo_{\mm}} \circ \Theta^{-1} 
= \text{the law of $\sum_{j=\mm}^{\infty}\delta_{Y_{j+1}}$},  
\]
which implies \eqref{:44e}. 
\qed\end{proof}

\section{Properties of kernel $\K^{\nn}$ and its Palm kernel
 $\K^{\nn}_{\x}$}\label{s:50}
We recall the definition of Schur functions following \cite{Mac}. 
Let $\delta = \delta_n = (n-1,n-2,\dots,1,0)$ and $\la =
(\la_1,\dots,\la_n)$ be a weakly decreasing sequence, i.e., 
$\la_1 \ge \la_2 \ge \dots \ge \la_n \ge 0$.  
The weight of $\la$ is defined by $|\la| = \sum_{i=1}^n \la_i$. 
For $\x = (x_1,\dots,x_n)$, we define 
\begin{equation}
 a_{\la + \delta}(\x) = \det(x_i^{\la_j + n - j})_{i,j=1}^n. 
\notag
\end{equation}
In particular, $a_{\delta}(\x)$ is the Vandermonde determinant given by 
\begin{equation}
 a_{\delta}(\x) = \det(x_i^{n-j})_{i,j=1}^n 
= \prod_{1 \le i < j \le n} (x_i-x_j). 
\label{:avan}
\end{equation}
The Schur function associated with $\la$ is defined by 
\begin{equation}
 s_{\la}(\x) = s_{\la}(x_1,\dots,x_n) 
= \frac{a_{\la+\delta}(\x)}{a_{\delta}(\x)}. 
\notag
\end{equation}
We note that every Schur polynomial is a linear combination of monomials 
with nonnegative integral coefficients. 

For later use, we give an estimate for sum of Schur functions 
\begin{lem}\label{l:121s}
For $\x \in \C^{\ell}$ and $i \in \n \cup \{0\}$, 
\begin{equation}
\sum_{\la : |\la| \le i} |s_{\la}(\x)|^2 
\le 
\Big( 
\prod_{j=1}^{\ell} (1+ |x_j|) 
\Big)^{2i}. 
\label{:121ss}
\end{equation}
\end{lem}
\begin{proof}
We recall the following formula for Schur functions (see, e.g., \cite{Mac}, 
Example~1, page 65). 
\begin{equation}
\sum_{\la : |\la| \le i} {i \choose \la} s_{\lambda}(\x) t^{|\la|} 
= 
\Big(
\prod_{j=1}^{\ell} (1+x_j t)
\Big)^i. 
\label{:121v} 
\end{equation}
Since the coefficients of Schur functions are all nonnegative and 
${i \choose \la} \ge 1$, 
by setting $t=1$ in \eqref{:121v}, we have 
\[
\sum_{\la : |\la| \le i} 
|s_{\la}(\x)|^2  
\le \Big(\sum_{\la : |\la| \le i} 
s_{\la}(|\x|)\Big)^2  
\le \Big( 
\prod_{j=1}^{\ell} (1+ |x_j|) 
\Big)^{2i},  
\]
where $|\x| = (|x_1|, |x_2|, \dots, |x_{\ell}|)$. 
\qed\end{proof}

Let $\phi_k(z) = z^k / \sqrt{k!}$. Note that 
$\K^{\nn}(z,w) = \sum_{k=0}^{\nn-1} \phi_k(z) \overline{\phi_k(w)}$. 
For $\x = (x_1,\dots, x_{\ell}) \in \C^{\ell}$ we set 
\[
\K^{\nn}(\x, \x) = (\K^{\nn}(x_i,x_j))_{i,j=1}^{\ell}, 
\]
$\K^{\nn}(z,\x) = (\K^{\nn}(z,x_1), \dots, \K^{\nn}(z,x_{\ell}))$ 
and 
$\K^{\nn}(\x, w) = \K^{\nn}(w,\x)^*$, where $*$ means
the complex-conjugation as matrix. 
From \pref{l:31}, if $\K^{\nn}(\x,\x)$ is invertible, 
or equivalently, $\det \K^{\nn}(\x,\x) >0$, 
then the Palm kernel $\K^{\nn}_{\x}$ has the following two
representations:  
\begin{align}
 \K^{\nn}_{\x}(z,w) 
&= \det 
\begin{pmatrix}
 \K^{\nn}(z,w) & \K^{\nn}(z, \x) \\
 \K^{\nn}(\x, w) & \K^{\nn}(\x,\x)
\end{pmatrix}
/ \det \K^{\nn} (\x,\x) 
\label{:121d}
\\
&=  \K^{\nn}(z,w) - \K^{\nn}(z,\x) 
\K^{\nn}(\x,\x)^{-1} \K^{\nn}(\x, w).  
\label{:121e}
\end{align}
The second equality is obtained from the well-known determinant 
formula for block matrix 
\[
\det 
\begin{pmatrix}
A_{11} & A_{12}\\
A_{21} & A_{22}
\end{pmatrix} = \det A_{22} \cdot
\det (A_{11} - A_{12} A_{22}^{-1} A_{21})
\quad \text{if $A_{22}$ is invertible}. 
\]
For $\x \in \C^{\ell}$ and $i \in \n \cup \{0\}$, 
we define a column vector by 
\begin{equation}
 \Phi_i(\x)
 = (\phi_i(x_1), \phi_i(x_2), \dots, \phi_i(x_{\ell}))^T
\notag
\end{equation}
and a Vandermonde-type $\ell \times \nn$ matrix by 
\begin{equation}
 V^{\nn}(\x) = 
(\Phi_{\nn-1}(\x), \Phi_{\nn-2}(\x), \dots, \Phi_{1}(\x), 
\Phi_{0}(\x))
\label{:121k}
.\end{equation}
Then, we have 
\begin{align}
 \K^{\nn}(\x, \x)
&= \sum_{p=0}^{\nn-1}  \Phi_p(\x)  \Phi_p(\x)^* 
= V^{\nn}(\x) V^{\nn}(\x)^*, 
\label{:121b} \\
\K^{\nn}(z, \x) 
&= \sum_{p=0}^{\nn-1} \phi_p(z) \Phi_p(\x)^*, \quad 
\K^{\nn}(\x, w)
= \sum_{p=0}^{\nn-1} \Phi_p(\x)\overline{\phi_p(w)}. 
\label{:121bb}
\end{align}

\begin{lem}\label{l:91}
Suppose $\nn > \mm$. 
The Palm kernel $\K^{\nn}_{\x}(z,w)$ for $\x \in \C^{\mm}$ admits 
the following representation 
\begin{equation}
 \K^{\nn}_{\x}(z,w) = q_{\x}(z) \overline{q_{\x}(w)}
 \Lsf^{\nn}_{\x}(z,w),  
\label{:Kq} 
\end{equation}
where $q_{\x}(z) = \prod_{j=1}^{\mm} (z-x_j) $ and 
$\Lsf^{\nn}_{\x}(z,w)$ is a Hermitian kernel function analytic in the
 first variable and is nowhere-vanishing on the diagonal.  
\end{lem}
\begin{proof}
Suppose $\nn > \mm$. We define an $(\ell+1) \times \nn$ matrix by 
\begin{equation}
V^{\nn}(z ; \x) = 
(\Phi_{\nn-1}(z,\x), \Phi_{\nn-2}(z,\x), \dots, \Phi_{1}(z,\x), 
\Phi_{0}(z,\x)), 
\notag
\end{equation}
where $\Phi_i(z, \x)
 = (\phi_i(z), \phi_i(x_1), \dots, \phi_i(x_{\ell}))^T$ 
for $z \in \C$ and $\x \in \C^{\mm}$. 
Then, by \eqref{:121d} and \eqref{:121b}, we have 
\begin{equation}
 \K_{\x}^{\nn}(z,w) 
= \det \Big( V^{\nn}(z ; \x) V^{\nn}(w ; \x)^* \Big) 
/ \det \Big( V^{\nn}(\x) V^{\nn}(\x)^* \Big)
\label{:ratio} 
\end{equation}
whenever $\det \K^{\nn}(\x, \x) > 0$. 
By the Binet-Cauchy formula, we have 
\begin{align} 
\lefteqn{\det \Big( V^{\nn}(z ; \x) V^{\nn}(w ; \x)^* \Big)
= \sum_{I \subset \{0,1,\dots, \nn-1\} \atop{|I|=\mm+1}} 
\det(V^{\nn}_I(z ; \x)) \overline{\det(V^{\nn}_I(w ; \x))}} \notag\\  
&= \sum_{I \subset \{0,1,\dots, \nn-1\} \atop{|I|=\mm+1}}  
\frac{1}{I!} s_{I - \delta_{\mm+1}}(z ; \x) 
\overline{s_{I - \delta_{\mm+1}}(w ; \x)} \cdot 
a_{\delta_{\mm+1}}(z ; \x) \overline{a_{\delta_{\mm+1}}(w ; \x)}   
\label{:zzx}
\end{align}
and 
\begin{align}
 \det \Big( V^{\nn}(\x) V^{\nn}(\x)^* \Big) 
&= \sum_{I \subset \{0,1,\dots, \nn-1\} \atop{|I|=\mm}} 
\det(V^{\nn}_I(\x)) \overline{\det(V^{\nn}_I(\x))} \notag  \\
&= \sum_{I \subset \{0,1,\dots, \nn-1\} \atop{|I|=\mm}}  
\frac{1}{I!} |s_{I - \delta_{\mm}}(\x)|^2 \cdot 
|a_{\delta_{\mm}}(\x)|^2. 
\label{zx}
\end{align}
Here $I! = \prod_{j=0}^{\mm} i_j!$ if 
$I = \{i_{\mm}, i_{\mm-1}, \cdots, i_1, i_0\}$ 
with $\nn -1 \ge i_{\mm} > i_{\mm-1} > \cdots > i_1 > i_0 \ge 0$, 
and $V^{\nn}_I(\x)$ is the $\mm \times \mm$ matrix obtained 
from the $\mm \times \nn$ matrix $V^{\nn}(\x)$ by deleting the columns 
corresponding to indices $\{0,1,\dots, \nn-1\} \setminus I$. 
The $(\mm+1) \times (\mm+1)$ matrix $V^{\nn}_I(z ; \x)$ is similarly obtained. 

From \eqref{:ratio}--\eqref{zx}, we can see that 
\eqref{:Kq} holds by setting 
\begin{equation}
\Lsf^{\nn}_{\x}(z,w) 
:= 
\sum_{I \subset \{0,1,\dots, \nn-1\} \atop{|I|=\mm+1}}  
\frac{s_{I - \delta_{\mm+1}}(z ; \x) 
\overline{s_{I - \delta_{\mm+1}}(w ; \x)}}{I!} \Big/  
\sum_{I \subset \{0,1,\dots, \nn-1\} \atop{|I|=\mm}}  
\frac{|s_{I - \delta_{\mm}}(\x)|^2}{I!} 
\label{:121m}
\end{equation}
since $a_{\delta_{\mm+1}}(z ; \x) /  a_{\delta_{\mm}}(\x) =
 \prod_{j=1}^{\mm}(z-x_j) = q_{\x}(z)$ by \eqref{:avan}. 

From \eqref{:121m}, it is clear that 
$\Lsf^{\nn}_{\x}(z,z) > 0$ for all $z \in \C$ when $\nn > \ell$. 
\qed\end{proof}

\begin{rem}\label{rem:zx} 
(1) From \eqref{:kmn} and \eqref{zx}, by letting $\nn \to \infty$, 
we see that the quantity in \eqref{:zx} 
is given by 
\[
Z(\x) = \frac{\det [\K(x_i,x_j)]_{i,j=1}^{\mm}}{|\Delta(\x)|^2}
= \sum_{I \subset \{0,1,2,\dots\} \atop{|I|=\mm}}  
\frac{1}{I!} |s_{I - \delta_{\mm}}(\x)|^2. 
\]
The right-hand side converges by the monotone convergence theorem. 
Since $s_{\mathbf{0}}(\mathbf{0}) = 1$ when $I = \delta_{\mm}$ and  
$s_{I - \delta_{\mm}}(\mathbf{0})=0$ otherwise, we  obtain
 \lref{Phiz}. \\
(2) Let $\x \in \C^{\ell}$. 
If $\nn \le \ell$, then $\K^{\nn}_{\x} = \Lsf^{\nn}_{\x} \equiv 0$, 
and if $\nn=\ell+1$, then $\Lsf^{\nn}_{\x}(z,w)$ 
does not depend on $z$ and $w$. Indeed, it follows from \eqref{:121m} 
that 
\[
\Lsf^{\ell+1}_{\x}(z,w) 
= 
\left(
\sum_{p=0}^{\ell} p! |s_{\la_{\ell, p}}(\x)| ^2
\right)^{-1}
= 
\left(
\sum_{p=0}^{\ell} p! |e_{\ell - p}(\x)|^2
\right)^{-1}, 
\]
where $\la_{\ell, p} = (\underbrace{1, \dots, 1}_{\ell - p}, 
\underbrace{0, \dots, 0}_{p})$, and 
$s_{\la_{\ell, p}}(\x)$ coincides with the $(\ell-p)$-th elementary symmetric 
function $e_{\ell-p}(\x)$. 
\end{rem}

In the next section, we will show a uniform variance estimate. 
For this purpose, we show an inequality for the difference between 
$\K^{\nn}$ and $\K^{\nn}_{\x}$. 

\begin{lem} \label{l:121p} 
Let $\x \in \C^{\ell}$. Then, for every $z \in \C$, 
\begin{equation}
(\K^{\nn}(z,z) - \K_{\x}^{\nn}(z,z))^{1/2}
\le \sum_{i=0}^{\nn-1} 
|V^{\ell}(\x)^{-1} \Phi_i(\x)| \, \phi_i(|z|),
\label{:123p}
\end{equation}
where $V^{\ell}(\x)$ is the $\ell \times \ell$ matrix 
as in \eqref{:121k} with $\nn=\ell$. 
When $V^{\ell}(\x)$ is not invertible, 
the right-hand side is understood as
\begin{equation}
 V^{\ell}(\x)^{-1} \Phi_i(\x) = \liminf_{\y \to \x, \y \in \mathcal{I}^{\ell}} 
V^{\ell}(\y)^{-1}
 \Phi_i(\y), 
\label{:123q} 
\end{equation}
where $\mathcal{I}^{\ell} := \{\y \in \C^{\ell} : \det V^{\ell}(\y) > 0\}$. 
\end{lem}
\begin{proof}
We note that $\K^{\nn}(\x, \x) \ge  \K^{\ell}(\x, \x)=  V^{\ell}(\x)
 V^{\ell}(\x)^{*}$ from \eqref{:121b}. 
Suppose $V^{\ell}$ is invertible. 
Then, by \eqref{:121e} and \eqref{:121bb}, we see that 
\begin{align*}
\K^{\nn}(z,z) - \K_{\x}^{\nn}(z,z)
&=\K^{\nn}(z,\x) \K^{\nn}(\x, \x)^{-1} \K^{\nn}(\x, z) \\
&\le \K^{\nn}(z,\x) (V^{\ell}(\x) V^{\ell}(\x)^*)^{-1} \K^{\nn}(\x, z) \\
&= |V^{\ell}(\x)^{-1} \K^{\nn}(\x, z)|^2 \\
&\le \Big(
\sum_{i=0}^{\nn-1} 
|V^{\ell}(\x)^{-1} \Phi_i(\x)| \, \phi_i(|z|) 
\Big)^2.  
\end{align*}
By continuity of $\K^{\nn}_{\x}(z,z)$ in $\x$, we obtain \eqref{:123p}
 with \eqref{:123q}. 
\qed\end{proof}

Although the right-hand side of \eqref{:123p} may diverge 
for $\x \in \mathcal{I}^{\ell}$ at this stage, 
we can see that it is finite indeed. 

\begin{lem} \label{l:121q} 
For $\x \in \C^{\ell}$ and $i \in \n \cup \{0\}$, 
\begin{equation}
 |V^{\ell}(\x)^{-1} \Phi_i(\x)|^2
= 
\begin{cases}
 1 & i =0,1, \dots, \ell-1, \\
\sum_{p=0}^{\ell-1} 
\frac{p!}{i!} |s_{\la_{i, \ell, p}}(\x)|^2 
& i \ge \ell, 
\end{cases}
\notag
\end{equation}
where $\lambda_{i, \ell, p}
= (i - \ell +1, \underbrace{1, \dots, 1}_{\ell-1-p}, 
\underbrace{0, \dots, 0}_{p})$. 
The left-hand side can be understood as a continuous extension in $\x$ 
even when $V^{\ell}(\x)$ is not invertible. 
\end{lem}
\begin{proof}
First suppose $\x \in \mathcal{I}^{\ell}$ as in \lref{l:121p}.  
For $p=0,\dots, \ell-1$, the $(\ell-p)$-th element of the column vector 
$V^{\ell}(\x)^{-1} \Phi_i(\x)$ is given by Cramer's rule:
\begin{align}
(V^{\ell}(\x)^{-1} \Phi_i(\x))_{\ell -p}
&= 
\frac{(\Phi_{\ell-1}(\x), \dots, 
\Phi_{p+1}(\x), \Phi_i(\x), \Phi_{p-1}(\x), 
\dots, \Phi_0(\x))}{(\Phi_{\ell-1}(\x), \Phi_{\ell-2}(\x),
 \dots, \Phi_0(\x))} 
\notag\\
&= 
\begin{cases}
\delta_{i,p} & i \le \ell-1, \\ 
(-1)^{\ell-1-p} (\frac{p!}{i!})^{1/2} s_{\la_{i, \ell, p}}(\x) 
& i \ge \ell. 
\end{cases}
\label{:123r}
\end{align}
Summing up these equalities for $p=0,1,\dots, \ell-1$ yields the
 assertion for $\x \in \mathcal{I}^{\ell}$. 
Since the right-hand side of \eqref{:123r} is continuous in $\x$, we
 conclude that 
$\liminf_{\y \to \x, \y \in \mathcal{I}^{\ell}}$ in \eqref{:123q} can be
 replaced by $\lim_{\y \to \x, \y \in \mathcal{I}^{\ell}}$. 
Therefore, $V^{\ell}(\x)^{-1} \Phi_i(\x)$ can be understood as 
a continuous function on $\C^{\ell}$. 
\qed\end{proof}

\begin{lem}\label{l:125a}
Let $C(\x) := \prod_{j=1}^{\ell} (1+|x_j|)$ for $\x \in \C^{\ell}$. 
Then, for each $i \in \n \cup \{0\}$
\[
 |V^{\ell}(\x)^{-1} \Phi_i(\x)| \le (\frac{(\ell-1)!}{i!})^{1/2}
 C(\x)^i. 
\]
\end{lem}
\begin{proof}
From \lref{l:121q} and \eqref{:121ss}, 
since $|\la_{i,\ell, p}| = i - p \le i$, 
\[
|V^{\ell}(\x)^{-1} \Phi_i(\x)|^2
= \sum_{p=0}^{\ell-1} \frac{p!}{i!}
|s_{\la_{i,\ell, p}}(\x)|^2
\le \frac{(\ell-1)!}{i!} \sum_{\la : |\la| \le i} 
|s_{\la}(\x)|^2
\le \frac{(\ell-1)!}{i!} C(\x)^{2i}
\]
for $i \ge \ell$, and it is clear that 
$|V^{\ell}(\x)^{-1} \Phi_i(\x)|^2 = 1 \le \frac{(\ell-1)!}{i!} C(\x)^{2i}$ 
for $0 \le i \le \ell-1$. Hence, we obtain the assertion.  
\qed\end{proof}

\section{Uniform estimate for variances}\label{s:6new}
We will use the symbol $\nn$ for the number of particles in 
$\nn$-particle approximations $\G^{\nn}$ and $\G^{\nn}_{\x}$, respectively. 

For later use, we show uniform boundedness for the 
variances of $\nn$-particle approximation $\G^{\nn}_{\x}$ 
with $\nn \in \n \cup \{\infty\}$ and $\x \in \C^{\mm}$ 
for $\mm \in \{0\} \cup \n$. 
Hereafter, we set $ \G ^{\nn}= \G $ if $ \nn = \infty $ and 
$\G^{\nn}_{\x} = \G^{\nn}$ for $\x \in \C^0$. 

\begin{lem}\label{var-monotone} 
Let $h : (0,\infty) \to \C$ be a bounded measurable function 
and put $g_p(z) = h(|z|)
 (\frac{|z|}{z})^p$ for $p \in \z$. 
Let 
\begin{equation}
I_{\nn}(p) 
:= \var^{\G^{\nn}}(\bra \xis, g_p \ket) 
+ \var^{\G_{\mathbf{o}_{\nn}}}(\bra \xis, g_p \ket)
- \var^{\G}(\bra \xis, g_p \ket). 
\notag
\label{decomp}
\end{equation}
Then, 
\begin{equation}
 I_{\nn}(p) = 
\sum_{k= (\nn-|p|)_{+}}^{\nn-1} 
\frac{1}{k! (k+|p|)!}
\left|\int_0^{\infty} h(\sqrt{t}) t^{k+\frac{|p|}{2}} e^{-t} dt \right|^2 
\le |p| \cdot \|h\|_{\infty}^2, 
\label{inp}
\end{equation}
where $(\nn-|p|)_{+} = \max(\nn-|p|, 0)$. 
In particular, 
\begin{equation}
\sup_{\nn \in \n \cup\{\infty\}} 
\var^{\G^{\nn}}(\bra \xis, g_p \ket) 
\le  \var^{\G}(\bra \xis, g_p \ket)
+ |p| \cdot \|h\|_{\infty}^2. 
\notag
\label{decomp2}
\end{equation}
\end{lem}
\begin{proof} 
We deduce from \eqref{:31f} and \eqref{:40b} that 
\begin{align}\label{:51e}&
 \K(z,w) = \K^{\nn}(z,w) + \K_{\mathbf{o}_{\nn}}(z,w). 
\end{align}
We easily see that kernels $ \K ^{\nn}$ and $ \K $ have 
the reproducing property with respect to $ \gsf $ in the sense of 
\eqref{reproducing}. Then so is $ \K_{\mathbf{o}_{\nn}} $ by 
\lref{l:reproducing}. 
Hence, from \eqref{var0} and \eqref{:51e} for the first equality, 
\eqref{:40b} and \eqref{:31f} 
 with $\ell=\nn$ for the second, we obtain 
\begin{align*}
I_{\nn}(p)
&= 
2 \int_{\C^2} g_p(z) \overline{g_p(w)} \Re(\K^{\nn}(z,w)
 \overline{\K_{\mathbf{o}_{\nn}}(z,w)}) \gsf(dz) \gsf(dw) \\
&= \sum_{k=0}^{\nn-1} \sum_{l=\nn}^{\infty} 
\frac{1}{k! l!} \delta_{l, k+|p|} 
\left|\int_{\C} h(|z|) |z|^{2k+|p|}\gsf(dz) \right|^2.  
\end{align*} 
This implies the equality in \eqref{inp}. 
By the Schwarz inequality, we obtain 
\[
\left|\int_0^{\infty} h(\sqrt{t}) t^{k+\frac{|p|}{2}} e^{-t} dt \right|^2
\le \|h\|_{\infty}^2 k! (k+|p|)!, 
\]
which implies the inequality in \eqref{inp}.
\qed\end{proof}

\begin{lem}\label{l:122x} 
Let $C(\x) = \prod_{j=1}^{\ell} (1+ |x_j|)$ for $\x \in \C^{\ell}$. 
Then, for all $z, w \in \C$, 
\begin{equation}
\sup_{\nn \in \n \cup \{\infty\}, \nn > \ell} 
|\K^{\nn}(z,w) - \K_{\x}^{\nn}(z,w)| 
\le (\ell-1)! \cdot e^{C(\x) (|z| + |w|)}. 
\label{:122a} 
\end{equation}
\end{lem}
\begin{proof}
Suppose $\det \K^{\nn}(\x,\x) > 0$. By \eqref{:121e}, we have 
\begin{equation}
 \K^{\nn}(z,w) - \K_{\x}^{\nn}(z,w)
= 
\K^{\nn}(z, \x)
\K^{\nn}(\x, \x)^{-1} \K^{\nn}(\x, w). 
\notag
\end{equation}
Since the right-hand side is a quadratic form, 
by the Schwarz inequality, we have 
\begin{equation}
| \K^{\nn}(z,w) - \K_{\x}^{\nn}(z,w)|
\le |\K^{\nn}(z,z) - \K_{\x}^{\nn}(z,z)|^{1/2} \cdot |\K^{\nn}(w,w) -
 \K^{\nn}_{\x}(w,w)|^{1/2}. 
\label{:121aa}
\end{equation}
By the continuity of $\K^{\nn}$ in $\x$, 
the above inequality holds for every $\x \in \C^{\ell}$. 

On the other hand, by \lref{l:121p} and \lref{l:125a}, we have 
 \begin{align*}
(\K^{\nn}(z,z) - \K_{\x}^{\nn}(z,z))^{1/2}
&\le \sum_{i=0}^{\nn-1} |V^{\ell}(\x)^{-1} \Phi_i(\x)| \, \phi_i(|z|)\\
&\le \sum_{i=0}^{\infty} (\frac{(\ell-1)!}{i!})^{1/2} C(\x)^i 
\cdot \frac{|z|^i}{(i!)^{1/2}}= ((\ell-1)!)^{1/2} e^{C(\x) |z|}. 
\end{align*}
Combining this with \eqref{:121aa} yields the desired uniform estimate.
\qed\end{proof}

\begin{rem}\label{rem:1cor}
Setting $z=w$ in \eqref{:122a} yields 
\begin{equation}
\sup_{\nn \in \n \cup \{\infty\}, \nn > \mm}
|\tilde{\rho}_1^{\nn}(z) - \tilde{\rho}_{1,\x}^{\nn}(z)| 
\le (\ell-1)! e^{C(\x)^2} e^{-(|z| - C(\x))^2}, 
\label{:1cor}
\end{equation}
where $\tilde{\rho}_1^{\nn}(z) = \gsf(z)^{1/2} \K^{\nn}(z,z)
 \gsf(z)^{1/2}$ is the $1$-correlation function of $\G^{\nn}$ with respect to 
the Lebesgue measure. 
\end{rem}

\begin{lem}\label{close2palm}
Let $f : \C \to \C$ be a measurable function. 
Then, for $\x \in \C^{\mm}$, 
there exists a positive constant $\tilde{C}(\x) >0$ such that 
\begin{equation}
\sup_{ \nn \in \n \cup \{\infty\},\, \nn > \mm }
| \var^{\G^{\nn}}(\bra \xis, f \ket)
- \var^{\G^{\nn}_{\x}}(\bra \xis, f \ket)| 
\le \tilde{C}(\x) \|f\|_{L^{\infty}(\C)}^2. 
\label{:420} 
\end{equation}
\end{lem}
\begin{proof} 
Let $\mathsf{I} (z,w)= |\K^{\nn} (z,w)|^2  - 
|\K^{\nn}_{\x}(z,w)|^2$. Then, 
from \eqref{var}, we see that  
\[
\var^{\G^{\nn}}(\bra \xis, f \ket)
- \var^{\G^{\nn}_{\x}}(\bra \xis, f \ket)
= \frac{1}{2}\int_{\C^2} |f(z)-f(w)|^2 \mathsf{I} (z,w) \pi^{-2}
 e^{-|z|^2-|w|^2} dzdw. 
\]
It is easy to see from \eqref{decrease}, \eqref{schwarz} and 
\lref{l:122x} that 
\[
|\mathsf{I} (z,w)|
\le 2 |\K^{\nn}(z,w) - \K^{\nn}_{\x}(z,w)|
\le 2 (\ell-1)! \cdot e^{C(\x) (|z| + |w|)}. 
\]
Combining these, we obtain \eqref{:420}. 
\qed\end{proof}

\section{Palm measures and singularity}\label{s:6}

In this section, we prove that $\G_{\x}$ and $\G_{\y}$ for 
$\x \in \C^{\mm}$ and $\y \in \C^{\nnn}$ are singular each other 
if $\mm \not= \nnn$. 
To this end, it is sufficient to show that so are 
$\G_{\oo_{\mm}}$ and $\G_{\oo_{\nnn}}$ for $\mm \not= \nnn$ since 
\[
 \G_{\x} \sim \G_{\oo_{\mm}} \perp  \G_{\oo_{\nnn}} \sim \G_{\y} 
\]
from the first part of Theorem~\ref{main1}.  
By \corref{:46a}, 
the point process over 
$[0,\infty)$ consisting of the square of 
moduli of $\G_{\oo_{\mm}}$ is equal in law to 
\begin{equation}
\eta_{\mm} := \sum_{i={\mm}+1}^{\infty} \delta_{Y_i}, 
\label{eq:eta_n}
\end{equation}
where $\{Y_i, i=1,2,\dots\}$ are as in \pref{kostlan}. 
That is,  the law of $\eta_{\mm}$ is equal to 
\begin{equation}
H_{\mm} := \G_{\oo_{\mm}} \circ \Theta^{-1} 
\quad (\mm \in \n \cup \{0\}). 
\notag
\end{equation}
Here $\{\eta_{\mm}\}_{\mm=0}^{\infty}$ are defined 
on a common probability space $(\Omega, \mathcal{F}, \Prob)$. 

If $H_{\mm}$ and $H_{\nnn}$ are singular each other, 
so are $\G_{\oo_{\mm}}$ and $\G_{\oo_{\nnn}}$, 
and then we first focus on the singularity of $H_{\mm}$'s 
instead of $\G_{\oo_{\mm}}$'s. 

\subsection{Ces\`aro mean of counting functions}

We define a function $\ft : Q([0,\infty)) \to \mathbb{R}$ for $T>0$ by 
\begin{align}\label{:FT}&
\ft(\eta) = \frac{1}{T} \int_0^T \{\eta([0,r]) -r \}dr.  
\end{align}
We remark that the function $\fft$ defined in \eqref{:13a} 
can be expressed as a lift of $\ft$ by $\Theta$, 
that is, $ \fft = \ft \circ \Theta$. 
We also note that 
\begin{equation}
 \frac{1}{T} \int_0^T \eta([0,r]) dr 
= \Big \bra \eta, \max(1-\frac{x}{T}, 0) \Big\ket.  
\notag
\end{equation}
Then it follows from \eqref{eq:eta_n} that 
for $\mm, \nnn \in \n \cup \{0\}$ with $\mm \le \nnn$, 
\begin{equation}
\ft(\eta_{\mm}) - \ft(\eta_{\nnn}) 
= \sum_{i=\mm+1}^{\nnn} \max\left(1-\frac{Y_i}{T}, 0\right). 
\label{FT1}
\end{equation}

\begin{lem}\label{P-as-f} 
For $\mm, \nnn \in \n \cup \{0\}$, 
\begin{equation}
\lim_{T \to \infty} \{\ft(\eta_{\mm}) - \ft(\eta_{\nnn})\} = \nnn - \mm 
\quad \text{$\Prob$-a.s.} 
\label{FT2}
\end{equation}
\end{lem}
\begin{proof} 
The assertion immediately follows from \eqref{FT1}. 
\qed\end{proof}

The next lemma, especially the estimate (\ref{bdd-var}) 
for the variance, is crucial for the proof of singularity. 
\begin{lem}\label{mean-var}
For each $\mm \in \n \cup \{0\}$, as $T \to \infty$,  
\begin{equation}
\E[\ft(\eta_{\mm})] \to - \mm, 
\label{mean}
\end{equation}
\begin{equation}
\var[\ft(\eta_{\mm})] = O(1). 
\label{bdd-var}
\end{equation}
\end{lem} 
\begin{proof}
Since $\E \eta_0([0,r]) = \sum_{i=1}^{\infty} \Prob(Y_i \le r) = r$, 
we have $\E[\ft(\eta_{0})]  =0$. 
On the other hand, we see from \eqref{FT1} that 
 $ \{ \ft(\eta_{\mm}) - \ft(\eta_{\nnn}) \}_{T>0} $ are bounded. 
Hence, from  (\ref{FT2}) and the bounded convergence theorem, 
 we deduce that for $\mm$ and $\nnn$ with $\mm \le \nnn$ 
\begin{align*}
\lim_{T \to \infty} 
\E[\ft(\eta_{\mm})  - \ft(\eta_{\nnn})] 
= \sum_{i=\mm+1}^{\nnn} 
\E \Big[ 
\lim_{T \to \infty} 
\max\left(1 - \frac{Y_i}{T}, 0 \right) \Big]
= \nnn - \mm.
\end{align*}
In particular, $\E[\ft(\eta_{\mm})] \to -\mm$. 
We have thus obtained \eqref{mean}. 

The proof of \eqref{bdd-var} will be given in the next subsection. 
\qed\end{proof}

\begin{lem}\label{P-as}
 We can choose $L^2(\Omega, \Prob)$-weak convergent subsequence 
$\ftk(\eta_{\mm})$ for $\mm \in \n \cup \{0\}$ with limits
$\flimit_{\mm}$: 
\begin{align}\label{:63a}&
 \lim_{k \to \infty} \ftk(\eta_{\mm}) = \flimit_{\mm} \quad 
\text{weakly in $L^2(\Omega, \Prob)$}. 
\end{align}
Moreover, we have 
\begin{equation}
\flimit_{\mm} - \flimit_{\nnn} = \nnn - \mm \quad \text{$\Prob$-a.s.}
\label{:63b}
\end{equation}
for any $\mm, \nnn \in \n \cup \{0\}$. 
\end{lem}
\begin{proof} 
For each $\mm$, $\{\ft(\eta_{\mm})\}_{T > 0}$  
is bounded in $L^2(\Omega, \Prob)$ 
by \eqref{bdd-var} and hence relatively compact weakly in $L^2(\Omega, \Prob)$.  
By a diagonal argument one can take convergent subsequences commonly in
 $\mm$. 
We have thus obtained the first claim. 
Since $|\ft(\eta_{\mm}) - \ft(\eta_{\nnn})| \le |\nnn- \mm|$ by \eqref{FT1}, 
by the dominated convergence theorem, (\ref{:63b}) 
follows from \eqref{FT2} and \eqref{:63a}. 
\qed\end{proof}

\begin{lem}\label{weakly}
 $\lim_{T \to \infty} \ft(\eta_{\mm}) = -\mm$ weakly in 
$L^2(\Omega, \Prob)$. In other words, 
$\lim_{T \to \infty} f_T = - \mm$  weakly in $L^2(Q([0,\infty)),
 H_{\mm})$. 
\end{lem}
\begin{proof}

We note that 
\[
\ft(\eta_0) = 
\sum_{i=1}^{\infty} 
\max(1 - \frac{Y_i}{T}, 0) - 
\sum_{i=1}^{\infty} \E[\max(1 - \frac{Y_i}{T}, 0)]
.\]
By the dominated convergence theorem, we see that for any $p \in \n$ 
\[
\sum_{i=1}^p 
\E\Big[
\Big(
\max(1 - \frac{Y_i}{T}, 0) - \E[\max(1 - \frac{Y_i}{T}, 0)]
\Big) 
Z \Big] 
\to 0 \quad \forall Z \in L^2(\Omega, \Prob), 
\]
from which we conclude that $\flimit_0$ is tail measurable. 
Hence, by Kolmogorov's $0$-$1$ law, 
$\flimit_0$ is constant $\Prob$-a.s. 
On the other hand, from (\ref{mean}) and \lref{P-as}
we see that $\E[\flimit_0] = \lim_{k \to \infty} \E[\ftk(\eta_0)] =
 0$, and then $\flimit_0 = 0$ $\Prob$-a.s. 
This together with (\ref{:63b}) for $\nnn=0$ 
yields $\flimit_{\mm} = - \mm$.  
Since $\flimit_{\mm} = - \mm$  is independent of the choice of 
a subsequence, \lref{weakly} follows 
immediately from Lemma~\ref{P-as} 
\qed\end{proof}

\begin{lem}\label{lem:weakimpliessingular}
Let $\mu$ and $\nu$ be probability measures on $X$. 
Suppose that the weak limits of $\{f_n\}_{n \ge 1}$
both in $L^2(X, \mu)$ and $L^2(X, \nu)$ exist and 
that they are constants $a$ and $b$, respectively. 
If $a \ne b$, then $\mu$ and $\nu$ are singular each other. 
\end{lem}
\begin{proof}
We write the Lebesgue decomposition as 
$\mu = h \nu + \eta$, where $h \in \Linu $ is nonnegative 
and $\eta$ is singular with respect to $\nu$. 
There exists a measurable set $E \subset X$ such that $\nu(E) = 0$ 
and $\eta(A) = \eta(A \cap E)$ for any $A$. 
Let $A_{\xpa} = E^c \cap \{h \le \alpha\}$. We note that 
$\eta(A_{\alpha}) = 0$ for all $\alpha>0$. 
Then, 
for every bounded measurable function $\phi$ on $X$, 
we see that 
\begin{align*}
\int_{A_{\alpha}} a \phi d\mu 
&= \lim_{n \to \infty} \int_{A_{\alpha}}  f_n \phi d\mu 
= \lim_{n \to \infty} \int_{A_{\alpha}}  f_n \phi (h d\nu + d\eta) \\
&= \int_{A_{\alpha}}  b \phi (h d\nu + d\eta)
= \int_{A_{\alpha}}  b \phi d\mu. 
\end{align*}
Hence, if $a \ne b$, then $\mu(A_{\alpha}) = 0$. 
By letting $\alpha \to \infty$, 
we have that $\mu(E^c \cap \{h < \infty\}) = 0$, or equivalently, 
\[
\mu(E \cup \{h = \infty\}) = 1. 
\]
On the other hand, since $\nu(\{h = \infty\}) = 0$, we see that 
\[
\nu(E \cup \{h = \infty\}) = 0. 
\]
Therefore, $\mu$ and $\nu$ are singular each other whenever $a \not= b$. 
\qed\end{proof}

Here we only show singularity part of \tref{main1}. 
We will prove absolute continuity part later in the proof of \tref{l:21}. 

\begin{proof}[of \tref{main1} (singularity) ]
It immediately follows from \lref{weakly} and
 \lref{lem:weakimpliessingular} that $H_{\mm}$ and $H_{\nnn}$ are 
singular each other whenever $\mm \not= \nnn$. 
Hence $\G_{\oo_{\mm}}$ and $\G_{\oo_{\nnn}}$ are mutually singular, from
 which the singularity of general Palm measures follows as mentioned 
in the beginning of this section. 
\qed\end{proof}

\subsection{Small variance property} 
In this subsection, we show (\ref{bdd-var}). 
We recall the following result obtained as Lemma 12 and Lemma 14 in
\cite{RV2}. Here we restate them in our situation. 
\begin{lem}\label{l:66}
Let $g : \C \to \C$ be of $C^1$ class and of compact support. 
Then, 
\begin{equation}
\var \big( \bra \xis, g(\frac{\cdot}{\rho}) \ket \big) 
\to \frac{1}{4\pi} \int_{\C} |\nabla g(z)|^2 dz
\label{:66x}
\end{equation}
as $\rho \to \infty$. For $g \in H^1(\C) \cap L^1(\C)$, 
\begin{equation}
\sup_{\rho > 1} \var \big( \bra \xis, g(\frac{\cdot}{\rho}) \ket \big) 
\le c \int_{\C} |\nabla g(z)|^2 dz
\label{:66y}
\end{equation}
for some universal constant $c>0$. 
\end{lem}

We remark that, if $g$ is rotationally invariant, i.e., 
$g(z) = h(|z|)$ for some $h : [0,\infty) \to \C$, \eqref{:66x} reads as 
\[
\var \big( \bra \eta_0, h(\frac{\cdot}{\rho})\ket \big) \to \frac{1}{2} \int_0^{\infty}
 x |h'(x)|^2 dx. 
\]
\begin{proof}[of (\ref{bdd-var}) ]
From \eqref{:13a} and \eqref{:FT}, 
if we take $g(z) = \max(1 - |z|^2, 0)$, i.e., $h(x) = \max(1-x^2, 0)$, 
and setting $g_T = g(\cdot/ \sqrt{T})$ and 
$h_T = h(\cdot/ \sqrt{T})$, we have 
\begin{align*}
 F_T(\xis) 
&= \bra \xis, g_T \ket 
- \E[\bra \xis, g_T\ket], \\ 
f_T(\eta_0) 
&= \bra \eta_0, h_T \ket
- \E[\bra \eta_0, h_T\ket]. 
\end{align*}
Hence, by \eqref{:66y}, 
\[
 \var(f_T(\eta_0)) = O(1). 
\]
From (\ref{FT2}) we observe that 
\[
 \var(\ft(\eta_{0}) - \ft(\eta_{\nnn})) 
= \sum_{i=1}^{\nnn} \var\left( \max(1 - \frac{Y_i}{T}, 0) \right)
\le \nnn \ (\nnn \in \n). 
\]
Therefore, we conclude that 
$\var[f_T(\eta_{\nnn})] = O(1)$ as $T \to \infty$ for every $\nnn \in \n$. 
\qed\end{proof}

\begin{rem}\label{r:71x}
Our test function $g(z) = \max(1- |z|^2, 0)$ is not $C^1$ but $H^1 \cap
 L^1$. Although \eqref{:66x} in \lref{l:66} cannot be applied directly, one can show that
 $\var(F_T) \to 1/2$ and it coincides with \eqref{:66x}. 
\end{rem}


Now we are in a position to prove \tref{l:13}.  
\begin{proof}[of \tref{l:13} ] 
We denote the Borel $\sigma$-field of $Q([0,\infty))$ by $\mathcal{R}$,
 and the law of $\eta_{\mm}$ 
on $(Q([0,\infty)), \mathcal{R})$ by $H_{\mm}$ as before. 
We recall that 
\[
 \fft = \ft \circ \Theta, \quad H_{\mm} = \G_{\oo_{\mm}} \circ
 \Theta^{-1}.  
\]
Then, $\var(\fft) (= \var(\ft))$ is uniformly bounded in $T$ 
under $\G_{\oo_{\mm}}$ from \eqref{bdd-var}, and thus 
a weak (subsequential) limit $\fflimit_{\mm}$ exists in
 $L^2(\G_{\oo_{\mm}})$.  
Moreover, since $\fft$ is $\Theta^{-1}(\mathcal{R})$-measurable, 
so is the weak limit $\fflimit_{\mm}$. Therefore, there exists 
an $\mathcal{R}$-measurable function $f \in L^2(Q([0,\infty), H_{\mm})$ 
such that $\fflimit_{\mm} = f \circ \Theta$. 
On the other hand, for every $\phi \in L^2(Q([0,\infty)), H_{\mm})$, 
\begin{align*}
\int f \phi dH_{\mm}
&= \int \fflimit_{\mm} \cdot (\phi \circ \Theta) d\G_{\oo_{\mm}} 
= \lim_{k \to \infty} 
\int \fftk \cdot (\phi \circ \Theta) d\G_{\oo_{\mm}}\\
&= \lim_{k \to \infty} \int \ftk \phi dH_{\mm}
= \int \flimit_{\mm} \phi dH_{\mm}, 
\end{align*}
which implies that $f$ must be $\flimit_{\mm}$. 
Consequently, from \lref{weakly}, 
$\fflimit_{\mm} = \flimit_{\mm} \circ \Theta = -\mm$,  
 which is the unique weak limit of $\{\fft\}_{T>0}$ in full sequence. 

For general $\x \in \C^{\mm}$, 
absolute continuity between $\G_{\oo_{\mm}}$ and $\G_{\x}$ yields 
the assertion for $\G_{\x}$ from \lref{sameweak} below. 
\qed\end{proof}

\begin{lem}\label{sameweak}
Let $\widehat{f}_{\mu}$ and $\widehat{f}_{\nu}$ be weak limits 
of $\{f_n\}$ in $L^2(X, \mu)$ and $L^2(X, \nu)$, respectively. 
If $\mu$ and $\nu$ are mutually absolutely continuous, 
then $\widehat{f}_{\mu} = \widehat{f}_{\nu}$ $\mu$-a.s. 
(or equivalently $\nu$-a.s.)
\end{lem}
\begin{proof}
Let $A_k = \{x \in X ; \frac{d\nu}{d\mu} \le k\}$ for $k \in \n$.  
For $\psi \in L^2(X, \nu)$ we set 
$\phi := \psi I_{A_k} \in L^2(X, \nu)$ 
for $k \in \n$. 
Since $\phi \frac{d\nu}{d\mu} \in L^2(X, \mu)$, we deduce that 
\begin{align*}
\int f_n \phi d\nu 
&= \int f_n \phi \frac{d\nu}{d\mu} d\mu
\to \int \widehat{f}_{\mu} \phi \frac{d\nu}{d\mu}d\mu
= \int \widehat{f}_{\mu} \phi d\nu. 
\end{align*}
This implies that $\widehat{f}_{\nu} = \widehat{f}_{\mu}$ $\nu$-a.e. on $A_k$ 
for any $k \in \n$. Since $\nu(\cup_{k \in \n} A_k) = 1$ when $\mu$ and
 $\nu$ mutually absolutely continuous, 
we conclude that $\widehat{f}_{\nu} = \widehat{f}_{\mu}$ $\nu$-a.s. 
\qed\end{proof}

\section{Absolute continuity of point processes} \label{s:7}

In this section, we will show a sufficient condition for two 
point processes to be mutually absolutely continuous in general setting. 

Let $ R $ be a complete separable metric space 
with metric $ \dist (\cdot, \cdot)$ and $\radon$ a Radon measure on $R $. 
We assume that $R$ is unbounded. 
We fix a point $ o \in R $ regarded as the origin, and set 
\begin{align}\label{:70a}&
S_r = \{ x\in R \, ;\, \dist (o,x) < b_r \} 
.\end{align}
Here $ \{ b_r \} $ is an increasing sequence of positive numbers such that 
$ \lim_{r \to \infty} b_r = \infty $. 
We will later choose $ \{ b_r \} $ suitably 
according to the model. 

Let $\QQQall $ be the configuration space over $R $, 
i.e., $\QQQall $ is a nonnegative integer-valued Radon measures on $R $ 
equipped with the vague topology. An element $\xis \in \QQQall $ can be expressed
as $\xis = \sum_i \delta_{s_i} $, and by definition, 
\begin{align}\notag 
\QQQall = \{ \xis = \sum_i\delta_{s_i}\, ;\, \xis (K) < \infty 
 \quad \text{ for all compact set } K \} 
.\end{align}
Let $ \QQQ $ be the subset defined by 
\begin{align}\notag 
\QQQ = \{ \xis \in \QQQall \, ;\, 
\xis (S_r ) < \infty \quad \text{ for all } r \in \mathbb{N}
\} 
.\end{align}
We say an element $ \xis$ of $ \QQQ $ a locally finite configuration. 
We note that this notion depends on the choice of the metric $ \dist $ 
equipped with $ R $.

Let $Q_{\nn} := \{\xis \in Q ; \xis(R) =\nn\}$ be 
the set of $\nn$-point configurations
and $Q_{\text{fin}} = \sqcup_{\nn=0}^{\infty} Q_{\nn}$, where $Q_0$ is the
singleton of empty configuration $\emptyset$. 
Clearly, $ Q_{\text{fin}} \subset Q \subset \QQQall $.

For a function $f:Q_{\text{fin}} \to \C$, 
there exist a constant $f^0 \in \C$ and symmetric functions $f^{\nn} 
: R^{\nn} \to \C$ so that $f(\emptyset) = f^0$ and 
$f(\xis) = f^{\nn}(s_1,s_2,\dots,s_{\nn})$ for 
$\xis = \sum_{i=1}^{\nn} \delta_{s_i} \in Q_{\nn}, \nn=1,2,\dots$. 
We say that $f : Q_{\text{fin}} \to \C$ is continuous if so is $f^{\nn} :
R^{\nn} \to \C$ for every $\nn$. 
We note that this continuity does not imply that in the vague topology, and 
 is enough for our argument. 
We often omit the superscript $\nn$ and abuse the same notation 
$f(\x)$ for the function $f^{\nn}$ on $R^{\nn}$. 

For a Borel subset $ A $, 
let $\p( A )$ be the set of all Borel probability measures 
$ \mu $ on $\QQQall $ 
such that $ \mu ( A^c ) = 0 $. 
We naturally regard such a $ \mu $ as the probability measure on 
$ (A,\mathcal{B}(A))$. 

\begin{dfn}\label{d:51}
Fix a Radon measure $\radon$ on $R$. 
We say that $\mu \in \p(Q)$ has an $\radon$-approximating sequence 
if (i) there exists a sequence $\mu^{\nn} \in \p(Q_{\nn}), \nn \in \n$ 
such that $\{\mu^{\nn}\}$ converges weakly to $\mu$, and 
(ii) there exists a continuous function $f_{\mu} : Q_{\text{fin}} 
\to [0,\infty)$ such that 
\begin{align}\label{:71a}&
 \mu^{\nn}(d\x) = \ZZ ({\mu^{\nn}})^{-1} f_{\mu}(\x) \radon^{\otimes \nn}
 (d\x) \quad (\x \in R^{\nn})
,\end{align}
where 
\begin{align}\label{:71b}&
\ZZ ({\mu^{\nn}}) = \int_{R^{\nn}} f_{\mu}(\x) \radon^{\otimes \nn} (d\x) 
.\end{align}
The totality of such probability measures is denoted by $\p_{\radon}(Q)$. 
\end{dfn}

In \dref{d:51}, we implicitly assume that 
$0 < \ZZ(\mu^{\nn}) < \infty$. 

Throughout this section we fix two probability measures 
$\mu, \nu \in \p_{\radon}(Q)$ with $ \radon $-approximations 
$\{\mu^{\nn}\}$ and $\{\nu^{\nn}\}$, respectively. 
We define the function $ \hh : Q_{\text{fin}} \to [0,\infty]$ by 
\begin{align}\label{:71c} & 
\hh (\xis) =\hh _{\mu,\nu}(\xis) = \frac{f_{\mu}(\xis)}{f_{\nu}(\xis)}, 
\end{align}
where it is understood to be $\infty$ when $\hh (\xis)$ is not defined. 
Note that the domain of $ \hh $ is $ Q_{\text{fin}}$, not $ Q $. 
We remark that for $\xis \in Q_{\text{fin}}$ 
\begin{equation}\notag
 \frac{d\mu^{\nn}}{d\nu^{\nn}}(\xis) 
= \hh (\xis) \frac{\ZZ(\nu^{\nn})}{\ZZ(\mu^{\nn})}. 
\end{equation}
Our task is to extend the domain of $ \hh $ naturally to $ Q $ 
in such a way that the extended function $ \hhh $ 
is the Radon-Nikodym density $ d\mu / d\nu $. 

Let $ S_r $ be as in \eqref{:70a}. 
For $r>0$, we define $\pi_r, \pi_r^c : Q \to Q$ by 
\begin{align}\notag 
 \pi_r(\xis) = \sum_{s_i \in S_r } \delta_{s_i}, \quad 
 \pi_r^c(\xis) = \sum_{s_i \in S_r^c } \delta_{s_i} 
\quad \text{ for $\xis = \sum_i \delta_{s_i}\in Q$}
.\end{align}
For the function $\hh$ on $Q_{\text{fin}}$, we define 
$\hh_r : Q \to [0,\infty]$ and 
$\hh_r^c : Q_{\text{fin}} \to [0,\infty]$ by 
\begin{equation}
\hh_r(\xis) = \hh(\pi_r(\xis)), \quad  
\hh_r^{c}(\xis) = \hh(\pi_r^c(\xis)) 
\notag 
\end{equation}
for each $r>0$, respectively. 
Then for each $\xis \in Q_{\text{fin}}$ we see that 
\begin{align}\notag 
\hh _r(\xis) = \hh (\xis ) \quad \text{ for sufficiently large $r = r_{\xis}$}
.\end{align}
This relation is a key to define the extension $ \hhh $ from $\hh$, 
and is a crucial consistency in our argument. 

For a non-decreasing positive sequence $ \{a_k\}_{k\in \n}$ 
we define a subset of $Q$ by 
\begin{equation}
\Hcal_k = \{\xis \in Q \, ; a_k^{-1} \le \inf_{r \in \n}\hh _r(\xis) \le  
\sup_{r \in \n}\hh _r(\xis) \le a_k\}. 
\label{:71g}
\end{equation}
For 
$ \{\mu^{\nn}\}_{\nn}$ and $ \{\nu^{\nn}\}_{\nn}$ introduced 
in \thetag{A2} below, we set 
\begin{align}\label{:71h}&
 \mu^{\nn}_k := \mu^{\nn}(\cdot | \Hcal_k) 
= \frac{\mu^{\nn } (\cdot \cap \Hcal_k)}{\mu^{\nn } (\Hcal_k)} 
,\quad 
 \nu^{\nn}_k := \nu^{\nn}(\cdot | \Hcal_k) 
= \frac{\nu^{\nn } (\cdot \cap \Hcal_k)}{\nu^{\nn } (\Hcal_k)} 
.\end{align}
Taking \thetag{A3}  below into account, 
we can and do assume that 
$ \mu (\Hcal_k)$, $ \nu (\Hcal_k)$, 
$\mu^{\nn } (\Hcal_k)$, and $\nu^{\nn } (\Hcal_k)$ are 
all positive. Hence the measures in \eqref{:71h} 
are well defined for all $ k, \nn $. 

We now state the assumptions by using the notion of 
$ \radon $-approximating sequences introduced in \dref{d:51}. 
Hence let 
$  \p_{\radon} $ be as in \dref{d:51}, 
$ \hh =\hh _{\mu ,\nu }$ be as in \eqref{:71c},  and 
$ \Hcal_k $ be as above. 
For a map $ f $ on $ Q$, we denote by $ \mathrm{Dcp}(f)$ 
the discontinuity points of $ f $. 
We assume the following: 
\\[2mm]
\thetag{A1} 
$ \mathrm{Dcp}(\pi _r)$ and $ \mathrm{Dcp}(\hh _r) $
are $ \mu $ and $ \nu $ measure zero for each $ r \in \mathbb{N}$. 
\\[2mm]
\thetag{A2} 
$ \mu $  and $ \nu \in \p_{\radon}$ with 
$ \radon $-approximating sequences 
$ \{\mu^{\nn}\}_{\nn}$ and $ \{\nu^{\nn}\}_{\nn}$. 
\\[2mm]
\thetag{A3}
$\displaystyle \lim_{k \to \infty} \liminf_{\nn \to \infty} 
\mu^{\nn}(\Hcal_k) =
 \lim_{k \to \infty} \liminf_{\nn \to \infty} 
\nu^{\nn}(\Hcal_k) = 1. $ 
\\[2mm]
\thetag{A4} For each $ r \in \mathbb{N}$, $\hh (\xis) =\hh _r(\xis)\hh _r^c(\xis)$ 
for $\xis \in Q_{\text{fin}}$. 
\\[2mm] 
\thetag{A5}
For each $k \in \n$, 
\[
 \sup_{\nn \in \n} \int_Q |\hh _r ^c -1 | d\nu^{\nn}_k 
= o(1) \quad (r \to \infty). 
\] 

Before proceeding to the proof, we briefly sketch the structure of our proof. 
From \thetag{A2} and \thetag{A3} we deduce that 
$\{(\mu_k^{\nn}, \nu_k^{\nn}, \nu^{\nn}(\Hcal_k))\}_{\nn}$ are relatively
compact. Let 
$\{(\mu_k^{\nn'}, \nu_k^{\nn'}, \nu^{\nn'}(\Hcal_k))\}_{\nn'}$ 
be convergent subsequence with intermediate limits 
$\{(\mu_k, \nu_k, \alpha_k^{\nu})\}$ such that 
\[
\{(\mu_k^{\nn'}, \nu_k^{\nn'}, \nu^{\nn'}(\Hcal_k))\} 
\to 
\{(\mu_k, \nu_k, \alpha_k^{\nu})\}
\to (\mu, \nu, 1). 
\]
We note that the supports of $\{\nu_k^{\nn'}\}$ are singular each other
in $\nn'$ for each $k$. Hence we introduce probability measures 
$\mu_k^{\nn'} \circ \pi_r^{-1}$ and $\nu_k^{\nn'} \circ \pi_r^{-1}$ and prove 
the convergence of Radon-Nikodym density 
$d\mu_k^{\nn'} \circ \pi_r^{-1} / d\nu_k^{\nn'} \circ \pi_r^{-1}$.  
Our proof consists of the following steps: 
\[
\frac{d\mu_k^{\nn'} \circ \pi_r^{-1}}{d\nu_k^{\nn'} \circ \pi_r^{-1}} 
\xrightarrow[\nn' \to \infty]{(a)}
\frac{d\mu_k \circ \pi_r^{-1}}{d\nu_k \circ \pi_r^{-1}} 
\xrightarrow[r \to \infty]{(b)}
\frac{d\mu_k}{d\nu_k}
\xrightarrow[k \to \infty]{(c)}
\frac{d\mu}{d\nu}.  
\]

\begin{rem}\label{r:81x}
\thetag{1} 
 Assumption \thetag{A4} implicitly assumes that 
the Radon-Nikodym densities $ d\mu^{\nn} / d\nu^{\nn} $ 
come from the sum of one-body potentials. 
In our application, 
we set $\mu^{\nn}$ and $\nu^{\nn}$ as 
two Palm measures $\mu^{\nn}_{\x}$ and $\mu^{\nn}_{\y}$ of 
\[
\mu^{\nn}(d\s)
= \ZZ(\mu^{\nn})^{-1} 
e^{-\sum_{1 \le i < j \le \nn} \Psi(|s_i - s_j|)} 
\radon^{\otimes \nn}(d\s)
\]
for $\s \in R^{\nn}$ and $\x, \y \in R^{\mm}$. 
Then the Radon-Nikodym derivative of 
two Palm measures $\mu^{\nn + \mm}_{\x}$ and $\mu^{\nn + \mm}_{\y}$, 
because of cancellation of mutual two-body potential part between $\xis$, 
is given by 
\begin{equation}
\frac{d\mu_{\x}^{\nn + \mm}}{d\mu_{\y}^{\nn + \mm}}(\s)
\propto 
e^{-\sum_{j=1}^{\nn} \widetilde{\Psi}(s_j)}, 
\notag 
\end{equation}
where 
$\widetilde{\Psi}(s) = \sum_{i=1}^{\mm} \{\Psi(|x_i - s|) - \Psi(|y_i -
 s|)\}$.   
In the case of the Ginibre point process $\G^{\nn}$, 
the two-body potential $\Psi$ is a logarithmic potential, i.e.,  
$\Psi(t) = -2\log t$. 

Under this setting, we also see that 
\[
 \frac{d\mu_{\x,k}^{\nn+\mm} \circ (\pi_r^c)^{-1}}{d\mu_{\y,k}^{\nn+\mm} \circ
 (\pi_r^c)^{-1}}(\xis)  
\approx \hh_r^c(\xis) \propto 
e^{-\sum_{s_j \in S_r^c} 
\widetilde{\Psi}(s_j)}.  
\]
We note that \thetag{A5} controls 
the tail part of the above Radon-Nikodym densities 
uniformly in $\nn$. 
\\
\thetag{2} 
We use \thetag{A1}--\thetag{A3} for 
$ \xrightarrow[\nn' \to \infty]{(a)} $. 
We use \thetag{A5} to justify the procedure 
$ \xrightarrow[r \to \infty]{(b)} $. 
In fact, \thetag{A5} is used in \eqref{:77d} of \lref{l:77}.  
We also remark that the step 
$\xrightarrow[k \to \infty]{(c)}$ follows from \thetag{A3} 
and \thetag{A5}. 

\end{rem}

The following four lemmas also hold for $\mu^{\nn}_{k}$ 
under the assumptions \thetag{A1}--\thetag{A3} because 
these assumptions are symmetric in $ (\mu ^{\nn }, \mu)$ and $ (\nu ^{\nn }, \nu )$. 

\begin{lem} \label{l:72} 
$ \{\nu _k ^{\nn }\}_{\nn \in \n}$ is tight in $ \nn $ for all $ k \in\n $. 
\end{lem}
\begin{proof}
Since $\nu^{\nn} \to \nu$ weakly from \thetag{A2}, $\{\nu^{\nn}\}_{\nn \in \n}$ is tight, 
i.e., for any $\epsilon >0$ 
there is a compact set $K_{\epsilon}$ such that 
$\sup_{\nn \in \n} \nu^{\nn}(K_{\epsilon}^c) \le \epsilon$. 
From this combined with \eqref{:71h}, we have 
\begin{align}\notag 
\limsup_{\nn \to \infty} \nu_{k}^{\nn}(K_{\epsilon}^c) 
\le  &\frac{\epsilon}{\liminf_{\nn \to \infty} \nu^{\nn}(\Hcal_k)} 
. \end{align}
Therefore, $\{\nu^{\nn}_{k}\}_{\nn \in \n}$ is tight from \thetag{A3}. 
\qed\end{proof}

For two measures $ m_1 $ and $ m_2$ 
on a measurable space $ (\Omega , \mathcal{F})$, 
$ m_1 \le m_2 $ means $ m_1 (A) \le m_2 (A)$ for all $ A \in \mathcal{F}$. 
\begin{lem}\label{l:73} 
There exists an increasing sequence $\{\nn_p\}_{p\in\n}$ 
of natural numbers suct that, for all $ k \in \n $, 
the following limits exist: 
\begin{align}\label{:73b} 
\nu_k := &\limi{p}\nu^{\nn_p}_k  \quad \text{ weakly} 
,\\ \label{:73w} 
\alpha^{\nu}_k := &\lim_{p\to\infty} \nu^{\nn_p}(\Hcal_k) 
.\end{align}
The limiting probability measures 
$\{\nu_k\}_{k\in\n}$ and constants $ \{\alpha^{\nu}_k\}_{k\in\n}$ satisfy 
\begin{align}
\label{:73a}&
0 \le \alpha^{\nu}_k \le \alpha^{\nu}_l \le 1 \quad \text{ for all } k\le l,\quad 
\limi{k} \alpha^{\nu}_k = 1
,\\&\label{:73c}
 \alpha_k^{\nu} \nu_k \le \alpha_l^{\nu} \nu_l \le \nu \quad 
\text{ for all $k \le l$}
.\end{align}
Furthermore, 
$\nu_k \to \nu$ strongly as $ k\to\infty $ in the following sense: 
\begin{align}\label{:73z}&
\limi{k} \nu_k (A) = \nu (A) \quad \text{ for all } 
A \in \mathcal{B}(\mathsf{S}) 
.\end{align} 
\end{lem}
\begin{proof}
By tightness of $\{\nu_k^{\nn}\}_{\nn}$ combined with the diagonal argument, 
 one can take a common subsequence $\{\nn_p\}_p$ of $\n$ so that 
\eqref{:73b} and \eqref{:73w} hold for all $ k \in \n $. 

Recall that $ \Hcal_k \subset \Hcal_l$ for all $ k\le l $ by definition. 
From this we see 
$ 0\le \nu^{\nn}(\Hcal_k) \le \nu^{\nn}(\Hcal_l) \le 1$. 
From \eqref{:73w} we then obtain the inequalities in \eqref{:73a}. 
From this and \thetag{A3}, we deduce that $ \limi{k} \alpha^{\nu}_k = 1$. 
We thus obtain \eqref{:73a}. 

From $ \Hcal_k \subset \Hcal_l$ for all $ k\le l $ and 
\eqref{:71h}, we deduce that 
\begin{align}\label{:73d}&
\nu^{\nn}(\Hcal_k) \nu_{k}^{\nn} \le 
\nu^{\nn}(\Hcal_l) \nu_{l}^{\nn} \le \nu^{\nn}
.\end{align}
Recall that $ 0 < \nu^{\nn_p}(\Hcal_k) \le 1 $. 
Then \eqref{:73c} follows from \eqref{:73w}, \eqref{:73d}, and \thetag{A2} by
 taking the limit along $\{\nn_p\}_{p \in \n}$. 

We deduce from \eqref{:73a} and \eqref{:73c} that, for any Borel set $ A $,  
\begin{align}\label{:73e}& 
\limsup_{k \to \infty} \nu_k(A) \le \nu (A)
.\end{align}
Replacing $ A $ with 
$ A^c$ in \eqref{:73e}, we obtain 
\begin{align}\label{:73g}&
\liminf_{k \to \infty} \nu_k(A) = 
1 - \limsup_{k \to \infty} \nu_k(A^c) \ge 
1- \nu (A^c)= \nu (A)
.\end{align} 
From \eqref{:73e} and \eqref{:73g}, we conclude \eqref{:73z}. 
\qed\end{proof}
 
\begin{lem}\label{l:74}
For each $k \in \n$, $\nu_k(\Hcal_k) = 1$ and 
$\nu_k$ is absolutely continuous with respect to $\nu$ and $\nu_l$ for
 $l \ge k$. 
\end{lem} 
\begin{proof}
Absolute continuity part follows from \eqref{:73c}. 
 Let $ \mathcal{H}_{k,r}=\{ a_k^{-1} \le \hh _r(\mathsf{s}) \le a_k \}  $. 
We easily see that 
\begin{align}\notag 
\mathrm{Dcp} (1_{\mathcal{H}_{k,r}}) \subset \{ \hh _r(\mathsf{s}) = 
a_k^{-1} \} \bigcup \{ \hh _r(\mathsf{s}) =  a_k \} 
\bigcup \mathrm{Dcp} (\pi_r )
.\end{align}
Note that the cardinality of the set 
$ \{ a ; \nu (\{ \hh _r (\xis) = a \}) > 0 \}$ is at most countable.  
Hence, retaking $ a_k$ if necessary, we can assume without loss of generality that 
$ \nu ( \hh _r(\mathsf{s}) =  a_k^{-1}) = 
  \nu ( \hh _r(\mathsf{s}) =  a_k ) = 0 $ 
for all $ k$. Then $ \nu (\mathrm{Dcp} (1_{\mathcal{H}_{k,r}}))=0 $. 
From this, \eqref{:73a} and \eqref{:73c}, we deduce that  
$   \nu_k (\mathrm{Dcp} (1_{\mathcal{H}_{k,r}}) ) = 0 $. 
From this and the weak convergence of $\nu _{k}^{\nn _p} $ to $\nu_k$ together with 
$ \nu _{k}^{\nn }(\mathcal{H}_{k,r}) = 1 $ for all $ \nn $, 
we deduce that for each $k$
\begin{align}\label{:76s}&
\text{$ \nu_k (\mathcal{H}_{k,r}) = 1 $ for all $ r \in \n $.}
\end{align} 
From \eqref{:76s} we deduce $ \nu_k(\mathcal{H}_k )=1 $ because 
$ \mathcal{H}_k = \cap_{r=1}^{\infty} \mathcal{H}_{k,r}$. 
\end{proof}
 
\begin{rem}\label{r:74}
\thetag{1} 
It follows immediately from \eqref{:73c} that 
$\alpha_k^{\nu} \nu_k \le \nu(\Hcal_k) \nu(\cdot | \Hcal_k)$. 
Since both $\nu_k$ and $\nu(\cdot | \Hcal_k)$ are probability measures, 
$\alpha_k^{\nu} := \lim_{p \to \infty} \nu^{\nn_p}(\Hcal_k) \le
 \nu(\Hcal_k)$. Furthermore, 
if $\alpha_k^{\nu} = \nu(\Hcal_k)$, then $\nu_k = \nu(\cdot |
 \Hcal_k)$. 
\\
\thetag{2} Let $ k $ be fixed. 
In \lref{l:72} and \lref{l:73}, we proved that $ \{\nu _k ^{\nn }\}_{\nn
 \in \n}$ is tight in $ \nn $ for each $ k $ and 
we took a convergent subsequence $ \{\nu^{\nn_p}_k \}_{p\in\n } $ with limit $ \nu_k$. 
We remark that the whole sequence $ \{\nu _k ^{\nn }\}_{\nn \in \n}$, however, 
does not converge in general. Indeed, by definition, 
$$ \nu _k ^{\nn } (d\mathsf{s})= \frac{1}{\nu ^{\nn }(\Hcal_k)} 
1_{\Hcal_k} (\mathsf{s})\nu ^{\nn }(d\mathsf{s})
.$$  
Since the function $ 1_{\Hcal_k} $ is not continuous, 
$ \{\nu _k ^{\nn }\}_{\nn \in \n}$ does not necessarily converge weakly 
in general. Furthermore, the limit point $ \nu _k $ 
may depend on the choice of the subsequence 
$ \{\nu^{\nn_p}_k \}_{p\in\n } $. This fact makes our proof complicated. 

We remark that, if $ \nu (\partial \Hcal_k) = 0 $, then the convergence of 
$ \{\nu _k ^{\nn }\}_{\nn \in \n} $ follows from that of 
$\{\nu ^{\nn }\}_{\nn \in \n}$ to $ \nu $. 
It would not be easy to prove $ \nu (\partial \Hcal_k) = 0 $ in general 
because $Q$ is equipped with the vague topology. 
Indeed, $\partial \Hcal_k$ may be large in general as is seen 
in \lref{l:744} below. In the case of the Ginibre point process, 
one can construct such a configuration $\xit \in Q$ as in \lref{l:744}. 

In the proof of \lref{l:74}, we introduced $ \sigma [\pi_r]$-measurable sets 
$ \mathcal{H}_{k,r} $ 
such that $ \cap_{r=1}^{\infty}\mathcal{H}_{k,r} = \mathcal{H}_k $ 
and that $ \nu (\partial \mathcal{H}_{k,r}) = 0 $, and 
applied the above mentioned strategy 
to $ \frac{1}{\nu^{\nn}(\mathcal{H}_{k,r})} 1_{\mathcal{H}_{k,r}}\nu ^{\nn} $. 

The point is that the limit $ \widetilde{\nu}:=\lim_{k\to\infty} \nu _k
 $ is independent 
of the choice of a subsequence $ \{\nu^{\nn_p}_k \}_{p\in\n } $ 
and coincides with $ \nu $. 
To prove this, we use the monotonicity argument in the sequel. 
The monotonicity comes from \eqref{:73a}, \eqref{:73c}, and
the definition of  $\{\Hcal_k\}$. 
\end{rem}

\begin{lem}\label{l:744}
Assume \thetag{A4}. 
If there exists a $ \xit \in Q $ 
such that $ 0 < \hhr (\xit) < \infty $ 
for all $ r \in \mathbb{N}$ and 
$\lim_{r\to\infty} \hhr(\xit) = \infty $, 
then $ \partial \Hcal_k = \Hcal_k$. 
\end{lem}
\begin{proof}
For each $ \xis \in \mathcal{H}_k $, $\xit \in Q$ as in the 
assumption, and $q > r > 0$, 
we set $ \xiu_{r,q} = \pi_r(\xis)+\pi_r^c (\pi_q(\xit)) \in
 Q_{\text{fin}}$. 
We see that from \thetag{A4} 
\[
\hh(\xiu_{r,q}) 
= \hhr(\pi_r(\xis)) \hhr^c(\pi_q(\xit)) 
= \hhr(\xis) \frac{\hhq(\xit)}{\hhr(\xit)}
= \frac{\hhr(\xis)}{\hhr(\xit)} \hhq(\xit). 
\]
This implies that $\lim_{q \to \infty} \hh(\mathsf{u}_{r,q}) = \infty$ 
for any fixed $r>0$ since 
$\hhr(\xis) / \hhr(\xit) > 0$ from the assumptions of $\xis$ and
 $\xit$. Hence 
$\lim_{q \to \infty} \sup_{s \in \n} \hh_s(\xiu_{r,q}) = \infty$ 
because $\hh(\xiu_{r,q}) \le \sup_{s \in \n} \hh_s(\xiu_{r,q})$.   
Therefore, for fixed $r>0$, there exists a $q(r) > r$ such that 
$a_k < \sup_{s \in \n} \hh_s(\xiu_{r,q(r)})$, 
which yields that $\xiu_{r,q(r)} \in \Hcal_k^c$. 
Since $\lim_{r \to \infty} \xiu_{r,q(r)} = \xis$ 
in the vague topology, 
we conclude that $ \partial \mathcal{H}_k = \mathcal{H}_k$.
\qed\end{proof}

\begin{lem}\label{l:75}
The Radon-Nikodym derivative ${d\nu_k}/{d\nu}$ converges to $1$ 
$\nu$-a.s. as $k \to \infty$. 
\end{lem}
\begin{proof} From \eqref{:73c} and \lref{l:74}, we have
 that 
\begin{equation}
\alpha^{\nu}_k \frac{d\nu_k}{d\nu}
\le \alpha^{\nu}_l \frac{d\nu_l}{d\nu} \le 1 
\quad \text{$\nu$-a.s.} 
\label{:75a}
\end{equation}
for $k \le l$. Since $\alpha^{\nu}_k \frac{d\nu_k}{d\nu}$ is 
non-decreasing in $k$ and $\alpha^{\nu}_k \nearrow 1$ 
by \eqref{:73a}, we see that 
$\lim_{k \to \infty} \frac{d\nu_k}{d\nu}$ 
exists and is bounded by $1$ $\nu$-a.s. 
By \lref{l:73} and the bounded convergence theorem, we obtain 
\begin{equation}
\int_Q \phi d\nu 
= \lim_{k \to \infty} \int_Q \phi d\nu_k
= \int_Q \phi \left(\lim_{k \to \infty} \frac{d\nu_k}{d\nu}\right) d\nu 
\notag
\end{equation}
for $\phi \in \Cb$. Therefore 
$\lim_{k \to \infty} \frac{d\nu_k}{d\nu} =1$ $\nu$-a.s. 
\qed\end{proof}
\begin{lem}\label{l:76}
Let $\{\nn_p\}_{p\in\n}$ be as in \lref{l:73}. For $  k ,  r \in \mathbb{N}$, 
we have 
\begin{align}
\label{:76b}&
\lim_{p \to \infty} \intQ \phi  \hh _r d\nu_k^{\nn_p }  
= \intQ \phi  \hh _r d\nu_k \quad \text{ for all $\phi \in C_{b}(Q) $}
.\end{align}
\end{lem}
\begin{proof}
From \thetag{A1} and \lref{l:74}, we obtain 
$ \nu_{k}(\mathrm{Dcp}(\hh _r)) =0$. 
This combined with the weak convergence of $\nu _{k}^{\nn _p} $ 
to $\nu_k$ yields \eqref{:76b}. 
\qed\end{proof}

As mentioned before, \lref{l:72}--\lref{l:75} hold 
for $\{\mu_k^{\nn}\}_{\nn \in \n}$ under the assumptions 
\thetag{A1}--\thetag{A3}. In particular, 
$\{\mu_k^{\nn}\}_{\nn \in \n}$ is tight in $ \nn $ for all $ k $. 
Therefore, by retaking a subsequence if necessary, 
we may assume that there exists
an increasing sequence $\{\nn_p\}_{p \in \n}$ of natural numbers such that 
$(\mu_k^{\nn_p}, \nu_k^{\nn_p}, \nu^{\nn_p}(\Hcal_k))$ 
converges to $(\mu_k, \nu_k, \alpha_k^{\nu})$ for all $k$ as $ p\to\infty$.  
In what follows, we use the symbol $\{\nn_p\}_{p \in \n}$ for such a sequence. 

\begin{lem}\label{l:77} 
Let $ \mu_k $ and $ \nu_k$ be 
the limits of $ \{ \mu_k^{\nn_p} \}_{p\in\n} $ and 
$ \{ \nu_k^{\nn_p} \}_{p\in\n} $, respectively. 
Then there exists a constant $\zeta_k \in (0,\infty)$ such that 
\begin{align}\label{:77a}&
\intQ \phi d\mu_k
= \zeta_k \lim_{r \to \infty} \intQ \phi\hh _r d\nu_k 
\quad \text{ for all $ \phi \in \Cb $}
.\end{align}
The constant 
$\zeta_k$ depends on $\{\nn_p\}_{p \in \n}$ such that  
\begin{align}\label{:77k}&\quad \quad 
\zeta_k = \lim_{p \to \infty} 
\frac{\ZZ ({\nu_k^{\nn_p }})}{\ZZ ({\mu_k^{\nn_p }})}
\quad \text{ for each $k \in \n$. }
\end{align}
\end{lem}
\begin{proof} Let $\iota_{\nn} : R^{\nn} \to Q$ be defined by 
$\iota_{\nn}(s_1,\dots,s_{\nn}) = \sum_{i=1}^{\nn} \delta_{s_i}$. 
We put $\Hcal_k^{\nn} = \iota_{\nn}^{-1}(\Hcal_k) \subset R^n$. 
We define 
\begin{equation}\notag
\ZZ ({\nu_k^{\nn} })
= \int_{\Hcal_k^{\nn }} f_{\nu}(\s) \radon^{\otimes \nn }(d\s), \quad  
\ZZ ({\mu_k^{\nn} })
= \int_{\Hcal_k^{\nn }} f_{\mu}(\s) \radon^{\otimes \nn }(d\s). 
\end{equation}
Then, we deduce from \eqref{:71a}--\eqref{:71c} that 
\begin{align*}
 \frac{\ZZ ({\nu_k^{\nn} })}{\ZZ ({\mu_k^{\nn }})}
&= \frac{\int_{\Hcal_k^{\nn }} f_{\mu}(\s) \{\hh (\s)^{-1}-1\}
\radon^{\otimes \nn }(d\s)}{\int_{\Hcal_k^{\nn }} f_{\mu}(\s) 
\radon^{\otimes \nn }(d\s)}  
+1 \\
&\le \max\{|a_k-1|, |1-a_k^{-1}|\} +1. 
\end{align*}
Hence there exists a subsequence 
$\{ \nn_s \}_{s\in\n}$ of $ \{ \nn_p \}_{p\in\n} $ 
with a limit $ \zeta_k $ such that 
\begin{align}\label{:77b}&
\lim_{s \to \infty} 
\frac{\ZZ ({\nu_k^{\nn_s }})}{\ZZ ({\mu_k^{\nn_s }})} = \zeta_k 
< \infty 
\end{align}
for each $k \in \n $. 
Interchanging the role of $ \mu_k $ and $ \nu_k $, 
we also see that $ \zeta_k > 0 $. 

We deduce from \thetag{A2} and \thetag{A4} that 
\begin{align} \label{:77c}
 \intQ \phi d\mu_k^{\nn } 
&= \frac{\ZZ ({\nu _{k}^{\nn }})}{\ZZ ({\mu_k^{\nn }})}
\intQ \phi\hh _r\hh _r^c d\nu^{\nn }_k 
\\ \notag 
&=   \frac{\ZZ ({\nu _{k}^{\nn }})}{\ZZ ({\mu_k^{\nn }})}
\{\intQ \phi\hh _r (\hh _r^c -1) d\nu^{\nn }_k 
+\intQ \phi\hh _r d\nu^{\nn }_k\}
.\end{align}
From \thetag{A5} and \eqref{:71g}, 
we obtain that, as $  r\to \infty $,  
\begin{align}\label{:77d}&
\limsup_{p \to \infty } 
| \intQ \phi\hh _r (\hh _r^c -1) d\nu^{\nn_p }_k | 
 \le  
\|\phi\|_{\infty}a_k o_r(1) 
.\end{align}
Since $\{ \nn_s \}_{s\in\n}$ is a subsequence of $ \{ \nn_p \}_{p\in\n} $, 
the sequences $ \{ \nu_k^{\nn_s }  \} $ and 
$ \{ \mu_k^{\nn_s }  \} $ converge to $ \nu_k $ and $ \mu_k $,
 respectively. 
Putting these, \lref{l:76}, \eqref{:77b}, and \eqref{:77d} into \eqref{:77c}, 
we obtain \eqref{:77a}. 

We note that $ \int_Q \phi d\mu_k $ and 
$ \lim_{r \to \infty} \int_Q \phi\hh _r d\nu_k$ are 
independent of the choice of the subsequence $ \{ \nn_s \}_{s\in\n} $ defining $ \zeta_k$. 
Hence, we deduce from \eqref{:77a} that 
$ \zeta_k$ is independent of the choice of $ \{ \nn_s \}_{s\in\n} $ 
(but may depend on $ \{ \nn_p \}_{p\in\n} $ at this stage). 
We have thus completed the proof. 
\qed\end{proof}

\begin{rem}\label{r:78}
By the definition of the normalization constants $\ZZ(\cdot)$'s, 
we see that $\ZZ(\nu^{\nn}) = \nu^{\nn}(\Hcal_k)^{-1}
 \ZZ(\nu_k^{\nn})$. 
By \lref{l:73} and \eqref{:77k}, we have 
\begin{equation}
  \lim_{p \to \infty} 
\frac{\ZZ(\nu^{\nn_p})}{\ZZ(\mu^{\nn_p})}
= \frac{\alpha_k^{\mu}}{\alpha_k^{\nu}} \zeta_k.  
\label{:78a}
\end{equation}
Note that the left-hand side of \eqref{:78a} does not depend on $k$, and
 hence so does the right-hand side. 
Since 
$\lim_{k \to \infty} \alpha_k^{\mu} = 
\lim_{k \to \infty} \alpha_k^{\nu}
 = 1$, we obtain from \eqref{:78a} that 
\begin{equation}
\zeta := \lim_{k \to \infty} \zeta_k 
= \lim_{p \to \infty} \frac{\ZZ(\nu^{\nn_p})}{\ZZ(\mu^{\nn_p})}
\label{:78b}
.\end{equation}
We will prove in \pref{l:7X} that 
 $ \zeta $ is independent of the choice of $ \{\nn_p\}_{p\in\n} $. 
\end{rem}

\begin{prop}\label{l:79}
Let $ \nu _{k} $ be as in \lref{l:73}. 
Then there exists a unique $ \hhi \in L^1(Q ,\nu )$ such that, for each $ k \in \mathbb{N}$,  
\begin{align}\label{:79a}&
\hhi = \lim_{r \to \infty} \hhr 
\quad \text{ weakly in } L^1(\nu _{k} )
.\end{align}
Furthermore, $\hh_{\infty}$ does not depend on the choice of 
a subsequence $ \{\nn_p\}_{p\in\n} $ 
and $\hh_\infty$ is not identically zero. 
\end{prop}
\begin{proof}
From \eqref{:71g} and $ \nu_k (\mathcal{H}_k) = 1  $ by \lref{l:74}, 
we deduce that, for all $ r \in \n $, 
\[
\|\hh _r \Hj \|_{L^{\infty}(\nu_k)} \le a_{\kk } 
.\]
We can take a common subsequence $\{r_m\}_{m} \subset (0,\infty)$ 
tending to $\infty$ such that, for all $ \kk \in \n$, the subsequence
$ \{ \hh _{r_m}\Hj \}_{m}$ is an $\Lone (\nu_k)$-weak 
convergent sequence with limit $ \hh _{\infty}^k\Hj $ 
satisfying $\|\hh _{\infty}^k\Hj \|_{L^{\infty}(\nu_k)} \le a_{\kk }$ 
for each $ \kk \in \n $. 
Here $ \hh _{\infty}^k $ is a function on $ Q $. 
From \eqref{:73c} we have that $ \Linu \subset L^{\infty}(\nu_k)$. 
These yield, for all $ \kk \in \n $, 
\begin{align}\label{:79b}&
\lim_{m \to \infty} \int \phi\hh _{r_m} \Hj d\nu_k 
= \int \phi\hh _{\infty}^k \Hj d\nu_k  
\quad \text{ for all $\phi  \in  \Linu   $}
.\end{align}
For $ k \le l $, we see that 
\begin{align}\label{:79c}
\int \phi\hh _{\infty}^k \Hj d\nu_k 
&= \lim_{m \to \infty} \int \phi\hh _{r_m} \Hj d\nu_k 
= 
\lim_{m \to \infty} 
\int \phi\hh _{r_m} \Hj 
\frac{d\nu_k}{d\nu_l}d\nu_l 
\\ \notag & = 
\int \phi\hh _{\infty}^l \Hj \frac{d\nu_k}{d\nu_l}d\nu_l 
= 
\int \phi\hh _{\infty}^l \Hj d\nu_k 
\quad \text{ for all } \phi \in \Linu   
.\end{align}
We used here $ \phi \frac{d\nu_k}{d\nu_l} \in L^{\infty}(\nu_l)$
 by the inequality 
$\frac{d\nu_k}{d\nu_l} \le \frac{\alpha^{\nu}_{l}}{\alpha^{\nu}_{k}}$, 
which follows from \eqref{:73c}. 
From \eqref{:79c}, we obtain the consistency such that, for all $ \kk \in \n $, 
\begin{align}\label{:79e}&\quad \quad 
\text{$ \hh _{\infty}^k  
= \hh _{\infty}^l  $ \quad \text{ for } $ \nu_k$-a.s.\ for all 
$ k \le l \in \n $}
.\end{align}
From \eqref{:73c} and $\nu_k(\Hcal_k)=1$ in \lref{l:74}, 
we deduce that $ \nu (\cup_{k=1}^{\infty}\mathcal{H}_{k} ) = 1 $. 
We then deduce from this and \eqref{:79e} that 
there exists a function $ \hh _{\infty} $ defined for $ \nu $-a.s.\ such that 
\begin{align*}& \quad \quad 
\text{$ \hh _{\infty} (\xis ) = \hh _{\infty}^k (\xis )  $ \quad 
 for $ \nu_k $-a.s.\ for all $ k \in \mathbb{N}$. }
\end{align*}
Hence we can rewrite \eqref{:79b} as 
\begin{align}\label{:79f}&
\lim_{m \to \infty} 
\int \phi\hh _{r_m} \Hj d\nu_k 
= \int \phi\hh _{\infty} \Hj d\nu_k  
\text{ for all $\phi  \in  L^{\infty}(\nu)   $}.
\end{align}
We deduce from \eqref{:79f} and \lref{l:77} that 
\begin{align}\label{:79g}&
\int \phi d\mu_k = \zeta_k \int \phi\hhi d\nu_k  
\quad \text{ for all $\phi  \in C_b(Q)   $}
.\end{align}
From \eqref{:79g}, 
the limit $ \hh _{\infty} $ in \eqref{:79f} is unique for $ \nu_k $-a.s. 
and, as a result, the whole sequence $\{\hh _r\}$ converges to 
$ \hh _{\infty} $ weakly in $L^1(\nu_k)$ for all $ k \in \mathbb{N}$. 
That is, 
\begin{align}\label{:79h}&
\limi{r} \int \phi\hh _{r} d\nu_k = \int \phi\hh _{\infty} d\nu_k  
\quad \text{ for all $\phi  \in  \Lnui  $}
.\end{align}
We remark that, at this stage, $\hhi$ in \eqref{:79h} may depend on 
$\{\nn_p\}_{p \in \n}$ defining $\{\nu_k\}_{k \in \n}$. Our next task is to 
prove that $\hhi$ does not depend on $\{\nn_p\}_{p \in \n}$. 

Let $\nu_k$ and $\nutilde_k$ be two subsequential limits, 
and $\hh_{\infty}$ and $\tilde{\hh}_{\infty}$ the corresponding 
limits of $\hh_r$ as in \eqref{:79h}, respectively. 
We then obtain from \eqref{:79h} that 
\begin{align}\label{:79i}&
\limi{r} \int \phi\hh _{r} d\nutilde_k 
= \int \phi\tilde{\hh}_{\infty} d\nutilde_k  
\quad \text{ for all $\phi  \in  \Lnui  $}
.\end{align}
For a fixed $\epsilon \in (0,1)$ and $k \in \n$, we let 
\begin{equation}
 A_k = \{\xis \in Q ; \eps \le \frac{d\nu_k}{d\nu} \le \eps^{-1},
  \ \eps \le \frac{d\nutilde_k}{d\nu} \le \eps^{-1} \}. 
\label{:79j} 
\end{equation}
From \lref{l:75}, we see that $\lim_{k \to \infty} \frac{d\nu_k}{d\nu}
= \lim_{k \to \infty} \frac{d\nutilde_k}{d\nu} =1$ $\nu$-a.s. 
 Hence, 
\begin{align}\label{:79p}&
\nu(\bigcup_{k \in \n} A_k ) =1
.\end{align}
For $\psi \in \Lnui $, set $\phi = \psi I_{A_k} \in \Lnui $. 
From \eqref{:79j} and $\phi = \psi I_{A_k} \in \Lnui $, we then see that 
\begin{align}\label{:79k}&
\phi \frac{d\nu_k}{d\nu} \left( \frac{d\nutilde_k}{d\nu} \right)^{-1} 
\in \Lnui 
.\end{align}
Hence, we deduce from \eqref{:79h}, \eqref{:79i}, and \eqref{:79k} that 
\begin{align} \notag 
 \int \hh_{\infty} \phi d\nu_k
&= \limi{r} \int \hh_r \phi d\nu_k \\\notag 
&= \limi{r} \int \hh_r \phi \frac{d\nu_k}{d\nu} 
\left( \frac{d\nutilde_k}{d\nu} \right)^{-1} d\nutilde_k \\\notag 
&= \int \tilde{\hh}_{\infty} \phi \frac{d\nu_k}{d\nu} 
\left( \frac{d\nutilde_k}{d\nu} \right)^{-1} d\nutilde_k 
= \int \tilde{\hh}_{\infty} \phi d\nu_k. 
\end{align}
This implies that 
$\hh_{\infty} = \tilde{\hh}_{\infty}$ $\nu$-a.e. on $A_k$ for 
any $k \in \n$. Therefore, from \eqref{:79p}, we conclude that 
$\hh_{\infty} = \tilde{\hh}_{\infty}$ $\nu$-a.s. 

Putting $\phi \equiv 1$ in \eqref{:79g}, we have 
$1 = \zeta_k \int \hhi d\nu_k. $
We have already seen that $\zeta_k \in (0, \infty)$. Therefore, 
$\hhi$ cannot be identically zero on $\Hcal_k$ for every $k$, 
and hence $\hhi$ is not identically zero by its definition. 
\qed\end{proof}

\begin{prop}\label{l:7X}
Let $ \zeta_k $, $ \zeta $, and $ \{\nn_p\}_{p\in\n} $ 
be as in \lref{l:77} and \rref{r:78}. 
Then $ \zeta $ is unique in the sense that $\zeta $ 
does not depend on the choice of a subsequence 
$ \{\nn_p\}_{p\in\n} $. 
Furthermore, $ \zeta $ is given by 
\begin{equation}
\zeta = \lim_{\nn \to \infty} 
\frac{\ZZ(\nu^{\nn})}{\ZZ(\mu^{\nn})}
\label{:7Xa}
.\end{equation}
\end{prop}
\begin{proof}
From \lref{l:74} and \eqref{:79g}, we see that 
\begin{align}\label{:7Xb}
\int \phi d\mu &= 
\limi{k}\int \phi d\mu_k 
= \limi{k} \frac{\zeta_k}{\alpha^{\nu}_k} 
\int \phi \hhi \big(\alpha^{\nu}_k\frac{d\nu _k}{d\nu } \big) d\nu 
.\end{align}
We see that 
$ \limi{k} {\zeta_k}/{\alpha^{\nu}_k}= \zeta $ from \eqref{:73a} and
\eqref{:78b} and that $\alpha^{\nu}_k\frac{d\nu _k}{d\nu } 
 \nearrow 1$ $\nu$-a.s. from \lref{l:75} with \eqref{:75a}. 
Hence, applying the monotone convergence theorem to the most 
right-hand side of \eqref{:7Xb}, 
we obtain that, for each non-negative $  \phi \in \Cb $, 
\begin{align}\label{:7Xc}
\int \phi d\mu = \zeta  \int \phi \hhi d\nu 
.\end{align}

From \pref{l:79}, we see that $\hh_{\infty}$ is 
independent of the choice of a subsequence $ \{\nn_p\}_{p\in\n} $ 
and uniquely determined $\nu$-a.s. 
Putting $\phi \equiv 1$ in \eqref{:7Xc} yields 
$1= \zeta  \int \hhi d\nu$. Since $\hhi$ was not identically zero from 
\pref{l:79}, we see that 
\[
 \zeta = \Big(\int \hhi d\nu \Big)^{-1} \in (0, \infty). 
\]
In particular,  $ \zeta $ is independent of the choice of a subsequence
$ \{\nn_p\}_{p\in\n} $. 

From \eqref{:78b} and the fact that $ \zeta $ is independent of 
$ \{\nn_p\}_{p\in\n} $, we conclude 
that $ \zeta $ is given by \eqref{:7Xa}. 
Indeed, from \eqref{:78b}, we deduce that, for any subsequence of $ \{\nn_i\}_{i\in\n}$, 
we can choose a further subsequence $ \{\nn_j\}_{j\in\n}$ such that 
$ \zeta = \lim_{j \to \infty} \frac{\ZZ(\nu^{\nn_j})}{\ZZ(\mu^{\nn_j})}$. 
This implies \eqref{:7Xa} because $ \zeta $ is independent of $ \{\nn_i\}_{i\in\n}$ and 
$ \{\nn_j\}_{j\in\n}$.  
\qed\end{proof}

The following theorem is the main result of this section. 
\begin{thm} 
\label{l:7J}
Under Assumptions \thetag{A1}--\thetag{A5}, 
$\mu $ is absolutely continuous with respect to $\nu $ with 
the Radon-Nikodym density 
\begin{align}\notag 
\frac{d\mu }{d\nu } = \zeta \hhi,  
\end{align}
where $\hh_{\infty}$ and $\zeta$ are given by 
\eqref{:79a} and \eqref{:7Xa}, respectively. 
\end{thm}
\begin{proof}
From \eqref{:7Xc}, we conclude that $ d\mu / d\nu = \zeta \hhi $. 
\qed\end{proof}

\section{Absolute continuity of Palm measures of Ginibre point process}
\label{s:8}

We verify Assumptions in \sref{s:7} 
for two Palm measures 
of the Ginibre point process $\G$ on $R=\C$ with the Euclidean metric. 
We take $ o $ as the origin of $ \C $ and $ b_r = 2^r$. 
Hence $ S_r = \{ z \in \C ; |z| < 2^r \} $ by definition. 
Let $ \radon $ be the Lebesgue measure on $ \C $. 

In this section we regard $\G$ (similarly for $\G^{\nn}$) as the DPP
associated with  $\tilde{\K}(z,w) = \pi^{-1} e^{z \wbar - \frac{1}{2}(|z|^2 +
|w|^2)}$ and the Lebesgue measure $\radon$ (see
\eqref{:gauge}).  
Now we verify Assumptions by setting $\x, \y \in \C^{\mm}$ and 
$\mu = \G_{\x}$ and $\nu = \G_{\y}$. Here $ \G_{\x}$ and $ \G_{\y}$ are 
reduced Palm measures of $ \G $ 
conditioned at $ \x $ and $ \y $, respectively. 
For $ \nn > \mm $ we set $ \G_{\x}^{\nn}$ and $ \G_{\y}^{\nn}$ 
by the reduced Palm measures of $ \G^{\nn} $ 
conditioned at $ \x $ and $ \y $, respectively. 

From (\ref{joint}), we consider the following function as $ \hh $ 
on $Q_{\text{fin}} = \sqcup_{\nn=0}^{\infty} Q_{\nn}$ 
\begin{align}\label{:80p}&
\hh (\xis) 
= \frac{\prod_{i=1}^{\nn} 
|\x - s_i|^2 }{\prod_{i=1}^{\nn} |\y - s_i|^2}
= \prod_{j=1}^{\mm} 
\frac{\prod_{i=1}^{\nn} 
|1 - x_j/s_i|^2 }{\prod_{i=1}^{\nn} |1 - y_j/s_i|^2 }
 \quad \xis \in Q_{\nn}. 
\end{align}

\begin{lem} \label{l:80}
\thetag{A1}, \thetag{A2}, and \thetag{A4} are satisfied with 
$ \mu = \G_{\x}$ and $ \nu = \G_{\y}$, 
$ \mu ^{\nn} = \G_{\x}^{\nn}$ and $ \nu ^{\nn}=\G_{\y}^{\nn}$. 
Moreover, $\hh $ is taken as \eqref{:80p}. 
\end{lem}

\begin{proof}
Since $ \nu $ has a locally bounded, 1-correlation function 
$\tilde{\rho}_1(\xis)$ with respect to the Lebesgue measure 
$ \radon $, \thetag{A1}, \thetag{A2} and \thetag{A4} are clear by construction. 
\qed\end{proof}

Let $ D_r = \{ z\in \mathbb{C}; |z| \le r \} $. 
Let $ \lceil z \rceil$ 
denote the minimal integer greater than or equal to $ z \in \mathbb{R}$. 
For fixed $\alpha \in \n$ we set 
\begin{align*} & 
u_{\alpha}(z) = \left(\frac{\lceil |z| \rceil}{z} \right)^{\alpha}, \quad 
v_{\alpha}(z) = \frac{1}{z^{\alpha}} 
,\\& 
F_{\alpha, r}(\xis) 
= \bra \xis, u_{\alpha} \ind_{D_r \setminus D_1}\ket
= \sum_{1 < |s_i| \le r} 
\left(\frac{\lceil |s_i| \rceil}{s_i}\right)^{\alpha} 
,\\& 
G_{\alpha, r,R}(\xis) 
= \bra \xis, v_{\alpha} \ind_{D_R \setminus D_r}\ket 
= \sum_{r < |s_i| \le R} \frac{1}{s_i^{\alpha}} , 
&& (\xis = \sum_{i} \delta _{s_i} )
.
\end{align*}

In the next two lemmas, 
we consider $\nu \in \p(Q)$ and 
a sequence $\{\nu^{\nn}\}_{\nn \in \n}$ of probability measures 
such that each $\nu^{\nn}$ is supported on $Q_{\nn}$
and $\{\nu^{\nn}\}_{\nn \in \n}$ is an $\radon$-approximation of $\nu$. 
Later, we will take $\nu^{\nn}$ to be the $ \nn $-particle approximation 
$\G_{\y}^{\nn}$ of 
Palm measures $\nu = \G_{\y}$ of the Ginibre point process. 

We introduce the notation  $ \nu^{\infty} := \nu $ to simplify the statement. 
\begin{lem}\label{l:81} 
Let $\alpha \in \n $. 
Assume that 
\begin{align}\label{:81a}&
\sup_{\nn \in \Ni } \|F_{\alpha, r}\|_{L^2(\nu^{\nn})} = O(r^c) 
\quad \text{ as $r \to \infty$ for some $c=c_{\alpha}<1$}.\end{align}
Then we obtain the following. 
\begin{align}\label{:81b}&
\sup_{\nn \in \Ni  \atop R>r} \|G_{\alpha, r,R} \|_{L^2(\nu^{\nn})} 
= O(r^{c-\alpha}).  
\end{align}
\end{lem}
\begin{proof}
For $\alpha \in \n$ and $r, R \in \n$ with $r < R$, 
by summation by parts, we have 
\begin{align*}
G_{\alpha, r,R}(\xis) 
&= \sum_{k=r+1}^R \frac{1}{k^{\alpha}} \sum_{k-1 < |s_i| \le k} 
\left(\frac{\lceil |s_i| \rceil}{s_i} \right)^{\alpha} \\
&= \sum_{k=r+1}^R \frac{1}{k^{\alpha}} 
(F_{\alpha, k} - F_{\alpha, k-1}) \\
&= \frac{F_{\alpha, R}}{R^{\alpha}} - \frac{F_{\alpha, r}}{(r+1)^{\alpha}} 
+ \sum_{k=r+1}^{R-1} \left(\frac{1}{k^{\alpha}} -
 \frac{1}{(k+1)^{\alpha}}\right) F_{\alpha, k}. 
\end{align*}
Hence, by taking the $ L^2(\nu^{\nn})$-norm and 
applying the assumption \eqref{:81a}, we obtain \eqref{:81b}. 
\qed\end{proof}

For $0 \le r < R < \infty $ and $ x \in \mathbb{C}$ we set 
\begin{align} \notag && 
\hht _{r, R} (\xis ) = \prod_{r < |s_i| \le R} | 1-\frac{x}{s_i} |^2 &&
\quad ( \xis = \sum_i \delta _{s_i} )
.\end{align}
By construction, we see that 
\begin{align}\label{:84y}&
\hhr (\xis ) = \frac{\prod_{j=1}^{\mm }\hhtt _{0,b_r}^{x_j }(\xis )}
{\prod_{j=1}^{\mm }\hhtt _{0,b_r}^{y_j }(\xis )}
.\end{align}

\begin{lem}\label{l:82} 
Assume that the $1$-correlation functions $ \tilde{\rho}_1^{\nn}$ of 
$\nu^{\nn}$ ($ \nn \in \Ni $) 
with respect to the Lebesgue measure on $\C$ satisfy 
\begin{align}\label{:82b}&
\sup_{\nn \in \Ni \atop{s \in \C}} \tilde{\rho}_1^{\nn}(s) < \infty 
.\end{align}
If there exists $c<1$ such that, for  $\alpha=1,2$, 
\begin{align}\label{:82a}&
\text{$\sup_{\nn \in \Ni } \|F_{\alpha, r}\|_{L^2(\nu^{\nn})} 
= O(r^c)$ as $r \to \infty$},  
\end{align}
then the following hold. \\
\thetag{1} For fixed $x \in \C$, 
\begin{align}\label{:82c}&
\sup_{\nn \in \Ni  \atop R>r} 
\|  \log\hht _{r,R}      
\|_{L^1(\nu^{\nn})} = O(r^{-(1-c)}) 
\quad \text{ as } 
r \to \infty
.\end{align}
\thetag{2} 
The series $ \log\hht _{0,r} $ 
converges in $ L^1(\nu ^{\nn})$ uniformly in $ \nn \in \Ni $ and 
uniformly in $x$ on any compact set in $\mathbb{C}$. 
Moreover, with $ b_p = 2^p$, 
\begin{align}\label{:84a}&
\lim_{k \to \infty} \inf_{\nn \in \Ni } 
\nu^{\nn} \Big( 
\sup_{p \in \n} |\log\hht _{0, b_{p}}| \le k\Big)  = 1. 
\end{align}
\end{lem}
\begin{proof}
For sufficiently large $r$ ($r \ge |x|+1$, say), we have 
\begin{align}& \label{:82f}
\left|\sum_{r < |s_i| \le R} \log (1-\frac{x}{s_i})
- \sum_{r < |s_i| \le R} \frac{x}{s_i}
- \frac{1}{2} \sum_{r < |s_i| \le R} \left(\frac{x}{s_i}\right)^2
\right| \le c_1
\sum_{r < |s_i| \le R} \left|\frac{x}{s_i}\right|^3 
.\end{align}
For the right-hand side, we deduce from \eqref{:82b} that 
\begin{align}\label{:82g}&
\sup_{\nn \in \Ni  \atop R>r} 
E^{\nu^{\nn}}[\sum_{r < |s_i| \le R} |\frac{x}{s_i} |^3 ]
\le 
\int_{D_r^c} \frac{|x|^3}{|s|^3} 
\{ \sup_{\Ni \atop{s \in \C}} \tilde{\rho}_1^{\nn}(s) \}
 ds = O(r^{-1}) 
.\end{align}
Hence from \eqref{:82f}, \eqref{:82g}, and \lref{l:81}, we have 
\begin{align}\notag 
\sup_{\nn \in \Ni  \atop R>r}
\| \log\hht _{r,R} \|_{L^1(\nu^{\nn})} 
& \le 
\max\big\{|x|, \frac{|x|^2}{2}\big\} 
\sum_{\alpha=1}^2 \sup_{\nn \in \Ni  \atop R>r} 
\|G_{\alpha, r,R}\|_{L^2(\nu^{\nn})} 
+ O(r^{-1}) 
\\ \notag 
& = O(r^{-(1-c)}). 
\end{align}
This completes the proof of \eqref{:82c}. 

Since 
\[
|\log\hht_{0,r}| = \Big|
\sum_{|s_i| \le r} \log |1 - \frac{x}{s_i}|^2 
\Big| 
\le 2 \left\bra 
\xis, 
\Big| \log |1 - \frac{x}{s}| \Big| \ind_{D_r}(s)
\right\ket, 
\]
we see from \eqref{:82b} that, for any $r>0$,  
\begin{align*}
\sup_{\nn \in \Ni} \|\log\hht_{0,r}\|_{L^1(\nu^{\nn})} 
&\le 2 \int_{D_r} \big| \log |1 - \frac{x}{s}| \big| 
\sup_{\nn \in \Ni} \tilde{\rho}_1^{\nn}(s) ds \\
&\le 2C \int_{D_r} 
\{\big| \log |s - x| \big| 
+ \big| \log |s| \big|\} 
ds < \infty. 
\end{align*}
Therefore, we deduce from \eqref{:82c} and the above that 
the series $ \{\log\hht _{0,r} \}_{r\in \mathbb{N}} $ 
is a Cauchy sequence in $ L^1(\nu ^{\nn})$ uniformly in $ \nn \in \Ni $. 
This yields the first claim of \thetag{2}. 

By a direct calculation we obtain that 
\begin{align}\label{:84b} 
 \nu^{\nn}&(\{ \sup_{p \in \n} |\log\hht _{0, b_{p}}| > k \})
=  \nu^{\nn}
\big(\bigcup_{p \in \n} 
\{|\log\hht _{0, b_{p}}| > k\} \big) 
\\ \notag & = 
  \nu^{\nn}
\big(\bigcup_{p \in \n} 
\{|\log\hht _{0, \infty } - \log\hht _{b_{p},\infty}| > k\} \big) 
\\ \notag &\le  
\nu^{\nn}
\big(
\{|\log\hht _{0, \infty } | > k/2\} \big) 
+
\nu^{\nn}\big(\bigcup_{p \in \n} 
\{|\log\hht _{b_{p},\infty}| > k/2 \} \big) 
\\ \notag & = 
\sum_{p=0}^{\infty} 
\nu^{\nn}(\{ |\log\hht _{b_{p},\infty }|> k/2 \}) 
 \quad \quad (\text{set $b_{0}=0$}) 
\\ \notag &\le 
 \frac{2}{k} \sum_{p=0}^{\infty} 
\|\log\hht _{b_{p},\infty}\|_{L^1(\nu^{\nn})} 
.\end{align}
We deduce from $b_p = 2^p$ and \eqref{:82c} that 
\begin{align}\label{:84c}&
\sup_{\nn \in \Ni } 
\sum_{p \in \n} 
\|\log\hht _{b_{p},\infty}\|_{L^1(\nu^{\nn})} 
< \infty 
.\end{align}
Therefore, we obtain from \eqref{:84b} and 
\eqref{:84c} that 
\begin{align}\notag 
\sup_{\nn \in \Ni }  \nu^{\nn}(\sup_{p \in \n} 
|\log\hht _{0, b_{p}}| > k)
= O(1/k) 
.\end{align}
\eqref{:84a} is immediate from this. 
We thus complete the proof of \lref{l:82}. 
\qed\end{proof}

\begin{cor}\label{cor:log}
Assumption \thetag{A3} holds for $\{\nu^{\nn}\}_{\nn \in
 \n \cup \{\infty\}}$ 
under the assumptions \eqref{:82b} and \eqref{:82a}. 
\end{cor}
\begin{proof}
Take $ a_k = e^{2\mm k}$. 
Then \thetag{A3} follows from \eqref{:84a} and \eqref{:84y}. 
\qed\end{proof}

We now proceed with the proof of \thetag{A3} and
\thetag{A5} for the Ginibre case. 
As before, we set $\x, \y \in \C^{\mm}$ and $\mu = \G_{\x}$ and $\nu =
\G_{\y}$. 
We also set $  \G^{\infty} = \G $,  $\G_{\x}^{\infty} =  \G_{\x} $, 
and $\G_{\y}^{\infty} =  \G_{\y}$.

We verify \thetag{A3}  by checking the assumptions
\eqref{:82b} and \eqref{:82a} in \corref{cor:log}.  

\begin{lem}\label{l:83}
For each $\x \in \C^{\mm}$, we have 
\begin{align}\notag 
\text{
$\sup_{\nn \in \Ni } \|F_{\alpha, r}\|_{L^2(\G_{\x}^{\nn} )} = O(r^{1/2})$ 
as $r \to \infty$. }
\end{align}
\end{lem}
\begin{proof} 
Notice that $\|F_{\alpha, r}\|_{L^2(\G_{\x}^{\nn} )} 
\le \{\var^{\G_{\x}^{\nn}}(F_{\alpha, r})\}^{1/2} +
 |\E^{\G_{\x}^{\nn}}[F_{\alpha,r}]|$. 
Since $\E^{\G^{\nn}}[F_{\alpha,r}] = 0$ by rotational invariance of
 $\G^{\nn}$,  
we can show that $|\E^{\G_{\x}^{\nn}}[F_{\alpha,r}]|=O(1)$ 
by using \eqref{:1cor}. 
Hence it suffices to show that 
$\sup_{\nn \in \n \cup \{\infty\}} \var^{\G_{\x}^{\nn}}(F_{\alpha,
 r}) =O(r)$. 

In \lref{var-monotone}, 
if we put $h(t) = \ind_{(1,r]}(t)(\frac{\lceil t \rceil}{t})^{\alpha}$ 
for $\alpha \in \n$, then $g_{\alpha} = F_{\alpha,r}$. 
Since $F_{\alpha, r}$ is uniformly bounded in $r$, 
it follows from \lref{var-monotone} and \lref{close2palm} that 
there exists a positive constant $c = c_{\x, \alpha}>0$ independent of
 $r$ such that 
\begin{equation}
\sup_{\nn \in \n \cup \{\infty\}} \var^{\G_{\x}^{\nn}}(F_{\alpha, r}) 
\le  \var^{\G}(F_{\alpha, r}) + c = O(r). 
\notag
\end{equation}
The last equality is a result from \cite[Theorem 1.2]{OS}. 
\qed\end{proof}

Since the $1$-correlation functions for 
$\{\G^{\nn}\}_{\nn \in \Ni}$ are uniformly 
bounded, so are 
those for $\{\G^{\nn}_{\x}\}_{\nn \in \Ni }$ 
by \eqref{:1cor} in \rref{rem:1cor}. 
Then \eqref{:82b} follows.  
Also \eqref{:82a} follows from \lref{l:83}. 
Therefore, \thetag{A3} holds for $\{\G_{\x}^{\nn}\}_{\nn
 \in \n \cup \{\infty\}}$ from \corref{cor:log}.  

We next verify \thetag{A5} for $\G_{\x, k}^{\nn} = 
\G_{\x}^{\nn}( \cdot | \Hcal_k)$ as in \eqref{:71h}. 

\begin{lem} \label{l:86}
For each $\x \in \C^{\mm}$, we have the following.  
\begin{align}\notag 
\sup_{\nn \in \Ni } 
\int_Q |\hh _r ^c -1 | d\G_{\x, k}^{\nn} = o(1) 
\quad \text{ as } r \to \infty 
.\end{align}
\end{lem}
\begin{proof}
Since $ \hh _r^c(\xis) = \lim_{s \to \infty} {\hh _s(\xis)}/{\hh _r(\xis)}$ 
for $\xis \in Q_{\text{fin}}$, we have 
\begin{align} \label{:86b}&
 a_k^{-2} \le \inf_{r \in \n}\hh _r^c(\xis)
\le \sup_{r \in \n}\hh _r^c(\xis) \le a_k^2 \quad \text{$\G_{\x, k}^{\nn} $-a.s.}
\end{align}
and, from \eqref{:84y}, 
\begin{equation}
 |\log \hh _r^c | \le 
\liminf_{s \to \infty} 
\sum_{j=1}^{\mm} 
\{| \log H_{b_r, b_s}^{x_j} | + | \log H_{b_r, b_s}^{y_j} |\}. 
\label{loghrc}
\end{equation}
Since $|t-1| \le (t \vee 1) |\log t|$ for $t > 0$ and 
\eqref{:86b}, we see that 
\begin{align}\label{:86c}&
\int_Q |\hh _r ^c -1| d\G_{\x, k}^{\nn}  \le \max\{a_k^2, 1\} 
\int_Q |\log\hh _r^c|  d\G_{\x, k}^{\nn}  
.\end{align}
In view of \lref{l:83}, 
applying \lref{l:82} \thetag{1} to \eqref{loghrc} 
and Fatou's lemma yield 
\begin{align}\label{:86d}&
\sup_{\nn \in \Ni } \int_Q |\log\hh _r^c|  d\G_{\x}^{\nn} 
= O(b_r^{-1/2})
.\end{align}
Since $\G_{\x,k}^{\nn} \le \G_{\x}^{\nn}(\Hcal_k)^{-1} \G_{\x}^{\nn} $, 
we obtain the assertion from \eqref{:86c} and \eqref{:86d}.  
\qed\end{proof}

The main theorem of this section is as follows. 
\begin{thm}	\label{l:87}
For each $\x ,\y \in \C^{\mm}$, $ \G_{\x }$ 
is absolutely continuous with respect to $ \G_{\y }$. 
The Radon-Nikodym density ${d\G_{\x }}/{d\G_{\y }}$ is given by 
\eqref{:21c} and \eqref{:21g}. 
\end{thm}
\begin{proof}
As we have seen above, all the assumptions of \tref{l:7J} are fulfilled. 
Hence from \tref{l:7J} we deduce the first claim. By \tref{l:7J} 
the Radon-Nikodym density is given by 
$ {d\G_{\x }}/{d\G_{\y }} = \zeta \hhi $, where 
$ \hhi = \limi{r} \hhr $ in the sense of \eqref{:79a} and 
$ \hhr $ on $ Q_{\text{fin}}$ is given by \eqref{:80p}. 
Hence it only remains to prove the convergence $ \hhi = \limi{r} \hhr $ is 
$ \G_{\y }$-a.s.\ by retaking the sequence $ \{b_r\}$ suitably, 
and to prove $ \zeta = 1/\ZZ_{\x \y } $. 

By \lref{l:82} \thetag{2}, we see that 
$ \log\hhi = \limi{r} \log\hhr $ converges  
in $ L^1 (\G_{\y })$ because $ \G_{\y } = \nu ^{\infty}$. 
Hence a suitable subsequence converges almost surely in $ \xis $, and 
uniformly on any compact set in $\C \backslash \{ y_i \}_{i} $ 
for $\G_{\y }$-a.e.\ $\xis = \sum_i \delta_{s_i}$.  

Hence \eqref{:21c} and \eqref{:21g} follows from \eqref{:7Xa} 
combined with \eqref{:42e}. 
\qed\end{proof}

\begin{proof}[of \tref{l:21} ]
\tref{l:21} follows from \tref{l:87} immediately. 
\qed
\end{proof}


\section{Concluding remarks and discussions}\label{s:10}

\textbf{1.} We have proved that 
the set of Palm measures (at finite points) 
of the Ginibre point process is decomposed into 
equivalence classes $\sqcup_{\mm=0}^{\infty} \mathcal{G}_{\mm}$ 
with respect to absolute continuity, where 
$\mathcal{G}_{\mm} = \{\G_{\x}, \x \in \C^{\mm}\}$.  
This result can be viewed as a generalization of the similar result for
finite point processes, i.e.,   
in the Ginibre case, $\infty -k  \not= \infty - \ell$ 
makes sense like $n-k \not= n - \ell$ for distinct $k$ and $\ell$. 
DPP can be thought to be associated with a reproducing kernel Hilbert
space; Ginibre is associated with the Bargmann-Fock space $H_{\K}$
with reproducing kernel being $\K(z,w) = e^{z \wbar}$, 
which is the space of $L^2$-entire
functions with respect to the complex Gaussian measure. 
The integral operator with kernel $\K$ is the projection operator 
from $L^2(\C, \gsf(z) dz)$ to $H_{\K}$. 
The Palm kernel $\K_{\x}$ for $\x = (x_1,\dots,x_{\mm})$ corresponds to 
the projection operator to the subspace $H_{\K} \ominus span\{\K(\cdot,
x_i), i=1,2,\dots, \mm\}$.   
From our results, the codimension of the range of $\K_{\x}$ in $H_{\K}$
can be detected from a typical configuration 
and subspaces of the same codimension are compared in 
terms of Radon-Nikodym density. 

One can also define a Palm measure $\G_{\x}$ for $\x \in
\C^{\infty}$, 
i.e., $\x$ is a countable subset of $\C$. 
It would also be interesting to discuss absolute continuity for such Palm
measures. 
In this direction, 
it is shown in \cite{Ghosh1,GP} that the conditioning of configuration 
outside of a disk determines 
the number of particles inside the disk $\G$-a.s, which is called {\it
rigidity} in Ginibre point process. 
From the above discussion, 
this might mean that 
the subspace $H_K \ominus\{\K(\cdot, x_i), i \in \n\}$ is finite dimensional.  
This is also related to a completeness problem of random exponentials
discussed in 
\cite{Ghosh2}. 

Absolute continuity between Poisson point processes 
can be completely determined by the Hellinger distance between their
intensity measures \cite{sk,T}; 
this is essentially due to Kakutani's dichotomy for absolute
continuity of infinite product measures \cite{kaku}. 
Absolute continuity for general DPPs seems more subtle 
from known results (\cite{HL,HS}) and it should also be discussed. 

\textbf{Question.} 
(i) Can one give a criterion for absolute continuity between 
(a class of) DPPs $\mu_{K,\la}$ in terms of $K$ and $\la$? \\
(ii) For fixed $K$ and $\la$, 
decompose $\{(\mu_{K,\la})_{\x}, \x \in R^{\mm}, \mm \in \n \cup
\{0\}\}$ into equivalence classes with respect to absolute continuity. 

\medskip
\noindent
\textbf{2.} 
We recall the notions of Ruelle's class potentials, 
canonical Gibbs measures and DLR equations. 
An interaction potential 
$ \Psi: \mathbb{R}^d \to \mathbb{R}\cup\{ \infty \} $ 
is called of Ruelle's class if $ \Psi $ is super stable in the sense of 
\cite{ruelle}, and regular in the following sense:   
There exists a positive decreasing function 
$ \psi :\mathbb{R}^+ \to \mathbb{R}$ and a constant $ R_0 >0 $ such that 
\begin{align}\label{:27B}&
\Psi (x) \ge - \psi (|x|) \quad \text{ for all } x ,
 \quad 
\Psi (x) \le \psi (|x|) \quad \text{ for all } 
  |x| \ge R_0 ,
\\ \notag  &
\int_0^{\infty} \psi (t)\,  t^{d-1}dt < \infty 
.\end{align}
A random point field $ \mu $ on $ \mathbb{R}^d$ is called 
a $ (\Phi , \Psi )$-canonical Gibbs measure 
if its regular conditional probabilities 
 $$ 
 \mu _{r,\xi }^{m} = 
 \mu (\,  \pi_r \in \cdot \, | \, 
 \pi_r^c (\mathsf{x}) = 
 \pi_r^c (\mathsf{\xi }),\, \mathsf{x}(S_r  ) = m )
 $$
satisfy the DLR equation. 
That is, for all  $r,m\in \mathbb{N}$ and $ \mu $-a.s.\! $ \xi $, 
\begin{align}\notag 
 \mu _{r,\xi }^{m} (d\mathsf{x}) =  
\frac{1}{\mathcal{Z}_{r,\xi}^m } 
e^{-\beta \mathcal{H}_{r,\xi}(\mathsf{x}) } 
\Pi_{\mathbf{1}_{S_r}dx}^{m} (d\mathsf{x}) 
.\end{align}
Here, 
$ \mathcal{Z}_{r,\xi}^m $ is the normalization, 
$ \Pi_{\mathbf{1}_{S_r}dx}^{m}= \Pi_{\mathbf{1}_{S_r}dx} 
(\cdot \cap \{ \mathsf{x}(S_r)=m \} )$ is the Poisson measure with intensity 
$\mathbf{1}_{S_r}dx$, and 
\begin{align}& \label{:A2y}
\mathcal{H}_{r,\xi}(\mathsf{x}) = \sum_{x_i \in S_r} \Phi (x_i) + 
\sum_{i<j,\,  x_i, x_j\in S_r, } \Psi (x_i-x_j) 
+ \sum_{x_i \in S_r,\, \xi_k \in S_r^c} \Psi (x_i-\xi_k)
\end{align}
If $ \mu $ is translation invariant, then $ \Phi = 0 $. 
From this and \eqref{:A2y}, $ \mu $ is loosely given by \eqref{:11ccc}. 

For $\gamma>0$ and $x \in \mathbb{R}^{\gamma}$, we define 
\begin{align}\label{:a1}&
\Psi_{\gamma} (x)= 
\begin{cases}
 \frac{1}{\gamma -2} |x|^{2-\gamma }
&(\gamma \not= 2)
\\
- \log |x|
& (\gamma = 2)
,\end{cases}
\end{align}
and its gradient is then given by 
\begin{align} \label{:a2}&
\nabla \Psi_{\gamma} (x)=  - \frac{x}{|x|^{\gamma }} 
.\end{align}
When $\gamma \in \n$,  
$ \Psi_{\gamma}$ is $\frac{\sigma_{\gamma}}{2}$ times 
the fundamental solution of $-\frac{1}{2}\Delta $
on $ \mathbb{R}^{\gamma }$, 
where $ \sigma_{\gamma} = 2\pi^{\gamma /2}/\Gamma (\frac{\gamma }{2})$ 
(the surface volume of the $ (\gamma -1) $-dimensional unit sphere
$\mathbb{S}^{\gamma-1}$). 
Since $ \Psi _{\gamma }$ gives the electrostatic potential 
in $ \mathbb{R}^{\gamma }$ ($ \gamma = 3$), 
 we call them Coulomb potentials.  
The sign of $ \Psi _{\gamma }$ is chosen 
in such a way that the potential describes the system of one component plasma. 

Here we call a translation invariant point process 
$ \mu _{\beta,\gamma,d}$ in $ \mathbb{R}^d$ 
a Coulomb point process if $ \mu _{\beta,\gamma,d}$ is 
a ``Gibbs measure'' with $ \gamma $-dimensional Coulomb 
potential $ \Psi _{\gamma }$ 
with inverse temperature $ \beta $, 
and called it a \textit{strict} Coulomb point process if $ \gamma = d $. 
 If $ d+2 < \gamma $, then $ \Psi _{\gamma }$ is a potential 
in the regime the classical theory can be applied to, and hence 
$\mu_{\beta, \gamma, d}$ can be constructed as a Gibbs 
measure via the ordinary DLR equation; however, if $ d \le \gamma \le d+2 $, 
then the integrability of the potential fails and 
the DLR equation does not make sense. 
So we need to construct $\mu_{\beta, \gamma, d}$ and understand it as a 
``Gibbs measure'' through the logarithmic derivatives 
discussed in Section 2. 
We characterize $\mu_{\beta, \gamma, d}$ using the logarithmic derivative
such that 
\[
 {\sf d}^{\mu_{\beta, \gamma, d}}(x, {\sf s}) 
= - \beta \lim_{r \to \infty} \sum_{|x - s_i| < r} 
\nabla \Psi_{\gamma}(x-s_i) 
= \beta \lim_{r \to \infty} \sum_{|x - s_i| < r}
\frac{x-s_i}{|x-s_i|^{\gamma}}. 
\]
We remark that, because $\mu_{\beta, \gamma, d}$ is translation
invariant, 
the sum converges in $L^1_{\mathrm{loc}}(\mathbb{R}^d \times Q, \mu_{\beta, \gamma,
d}^{1})$. 
Ginibre point process is, so far, the only example of 
strict Coulomb point processes rigorously constructed and verified 
as the ``Gibbs measure'' in the case of $ (\beta,\gamma,d) = (2,2,2)$.  

\textbf{Question.} 
\thetag{iii} 
It would be interesting to construct strict Coulomb point processes 
other than the Ginibre point process, and to prove the analogy of 
Theorem~\ref{main1} holds for all $ \beta > \beta_0$ for some 
 $ \beta_0\ge 0$, and if this is the case, then to prove or disprove 
that $ \beta_0 >0 $, which is the existence of a phase transition for this phenomena. 

\vskip 5mm
\noindent
{\bf Acknowledgment. }
The authors would like to thank the anonymous referee for the careful reading and 
helpful comments which improve our manuscript. 
The first author~(HO)'s work was supported in part 
by JSPS Grant-in-Aid for Scientific Research (A) No. 24244010 and 
(B) No. 21340031.  
The second author~(TS)'s work was supported in part 
by JSPS Grant-in-Aid for Scientific Research (B) No. 22340020 and 
(B) No. 26287019.

\end{document}